\documentclass[a4paper,12pt,reqno]{amsart}
\usepackage[english]{babel}
\usepackage{amsmath}
\usepackage{amscd}
\usepackage{amsgen}
\usepackage{amsfonts}
\usepackage{amssymb}
\usepackage{amsthm}
\usepackage{comment}
\usepackage{url}

\catcode`\@=11
\@namedef{subjclassname@2010}{%
\textup{2010} Mathematics Subject Classification} \catcode`\@=12

\newtheorem{theorem}{Theorem}
\newtheorem*{theorem*}{Theorem}

\newcounter{prop}
\newtheorem{prop}[theorem]{Proposition}

\newtheorem*{prop*}{Proposition}

\newtheorem{thm}{Theorem}
\newtheorem{lem}[theorem]{Lemma}
\newtheorem{cor}[theorem]{Corollary}

\newtheorem{defn}[theorem]{Definition}
\newtheorem{remark}[theorem]{Remark}
\newtheorem{example}[theorem]{Example}

\newcommand{\R}{\mathbb{R}}
\newcommand{\C}{\mathbb{C}}
\newcommand{\N}{\mathbb{N}}
\newcommand{\Z}{\mathbb{Z}}
\newcommand{\HS}{\mathbb{H}}
\renewcommand{\P}{\mathcal{P}}

\newcommand{\Ric}{\operatorname{Ric}}
\newcommand{\id}{\operatorname{Id}}
\newcommand{\tr}{\operatorname{tr}}
\newcommand{\hyp}{\operatorname{hyp}}
\newcommand{\Conf}{\operatorname{Conf}}
\newcommand{\ad}{\operatorname{ad}}

\newcommand{\abs}[1]{\lvert#1\rvert}
\newcommand{\grad}{\text{grad}}

\newcommand{\G}{{\mathcal G}}

\newcommand{\Res}{\operatorname{Res}}
\newcommand{\ID}{(I\!\cdot\!D)}
\newcommand{\IDD}{I\!\cdot\!D}

\def\S{{\Sigma}}            
\def\Sc{{\mathcal S}}           
\def\T{{\mathcal T}}            
\def\H{{\mathcal H}}            
\def\D{{\mathcal D}}            
\def\M{{\mathcal M}}            
\def\U{{\mathcal U}}            

\def\Rho{{\sf P}}               
\def\J{{\sf J}}                 

\def\m{m}
\def\st{\stackrel{\text{def}}{=}}

\numberwithin{theorem}{section} \numberwithin{equation}{section}
\numberwithin{prop}{section}

\begin{document}

\title{Shift operators, residue families and degenerate Laplacians}
\author[Andreas Juhl and Bent {\O}rsted]{Andreas Juhl and Bent {\O}rsted}

\thanks{This paper would not be without the major ideas and crucial
contributions of Matthias Fischmann; we thank him heartily.}

\address{Department of Mathematics of {\AA}rhus University, Ny Munkegade 118, 8000
{\AA}rhus, Denmark}


\email{ajuhl@math.hu-berlin.de}
\email{orsted@math.au.dk}


\allowdisplaybreaks


\begin{abstract}
In this paper, we introduce new aspects in conformal geometry of some very natural
second-order differential operators. These operators are termed shift operators. In
the flat space, they are intertwining operators which are closely related to
symmetry breaking differential operators. In the curved case, they are closely
connected with ideas of holography and the works of Fefferman-Graham, Gover-Waldron
and one of the authors. In particular, we obtain an alternative description of the
so-called residue families in conformal geometry in terms of compositions of shift
operators. This relation allows easy new proofs of some of their basic properties.
In addition, we derive new holographic formulas for $Q$-curvatures in even
dimension. Since these turn out to be equivalent to earlier holographic formulas,
the novelty here is their conceptually very natural proof. The overall discussion
leads to a unification of constructions in representation theory and conformal
geometry.
\end{abstract}


\maketitle

\begin{center} \today \end{center}

\tableofcontents

\section{Introduction and formulation of the main results}\label{intro}

Conformal differential geometry has seen spectacular developments in recent
years, both from a perspective of pure mathematics, and from a mathematical
physics point of view. The construction of ambient metrics and Poincar\'e
metrics by Fefferman and Graham \cite{FG} gave rise to many important
applications and fundamental insights. This is closely connected to the idea of
holography.

Attempts to extend these ideas to a conformal submanifold theory were one
source for the notion of symmetry breaking operators. This recent notion in
representation theory is central in studies of, for instance, the interplay
between the representation theory of the conformal group (the M\"{o}bius group) of
Euclidian space and the corresponding group for a hyperplane. In particular, it
plays a basic role in the study of branching laws of representations. The works
\cite{koss, KS1, KS2, KKP, FJS, MO} reflect recent progress in this area.
Curved analogs of symmetry breaking operators in conformal geometry deal with
conformally covariant differential operators $C^\infty(X) \to C^\infty(M)$,
where $M$ is a hypersurface of a Riemannian manifold $X$ (see \cite{J1} and
references therein). Residue families (introduced in \cite{J1}) are curved
versions of symmetry breaking operators which are defined in a setting where
$X$ is a tubular neighborhood of $M$ and where the metric on $X$ is determined
by the metric on $M$. For recent substantial progress in the general case we
refer to \cite{GP}.

In this paper, we shall develop a theory of some second-order differential
operators, originally found via representation theory as {\em shift operators}
between symmetry breaking operators. These operators turn out to be very
natural and admit generalizations within a framework defined by Riemannian
metrics. Remarkably, these generalizations did appear in earlier work from a
number of different perspectives, in particular from the point of view of
tractor calculus, a powerful tool in conformal geometry.

Shift operators recently appeared in the theory of symmetry breaking operators.
The latter operators are generalizations of Knapp-Stein intertwining operators
which intertwine principal series representations of semi-simple Lie group.
Symmetry breaking operators map between functional spaces on a given flag
variety to functional spaces on a subvariety and are equivariant only with
respect to the symmetry group of the subvariety. This loss of symmetry is the
origin of the notion. A typical situation is that of the round sphere $S^{n+1}$
with an equatorially embedded subsphere $S^n \hookrightarrow S^{n+1}$. The
non-compact model of that situation is a standard embedding $\R^n
\hookrightarrow \R^{n+1}$. In these cases, the relevant groups are the
conformal groups of the respective submanifolds. Conformal symmetry breaking
{\em differential} operators acting on functions in that setting are of
particular importance for conformal differential geometry \cite{J1}.

Shift operators shift the spectral parameter in the distributional Schwartz
kernels of conformal symmetry breaking operators. Such results in the setting
$\R^n \hookrightarrow \R^{n+1}$ first appeared in \cite{FOS} and will be
recalled in Section \ref{FlatShiftOperator}. The basic shift operator in that
theory is given by the $1$-parameter family \cite[(3.5)]{FOS}
\begin{equation}\label{flat-shift}
    P(\lambda) = r \Delta - (2\lambda-n-3) \partial_r: C^\infty(\R^{n+1}) \to
    C^\infty(\R^{n+1}), \; \lambda \in \C
\end{equation}
of second-order differential operators on $\R^{n+1}$. Here $\Delta$ denotes the
non-positive Laplacian of the flat metric on the space $\R^{n+1}$ with
coordinates $(r,x)$. We regard $r$ as a defining function of the subspace
$\R^n$. For any $\lambda \in \C$, the operator $P(\lambda)$ is equivariant with
respect to principal series representations restricted to the conformal group
$\Conf(\R^n)$ of the subspace $\R^n$ with the flat metric regarded as the
subgroup of the conformal group $\Conf(\R^{n+1})$ of $\R^{n+1}$ leaving $\R^n$
invariant. More precisely, $P(\lambda)$ satisfies the intertwining relation
\begin{equation}\label{flat-shift-equivariance}
    \left( \frac{\gamma_*(r)}{r}\right)^{n-\lambda+2} \circ \gamma_* \circ P(\lambda)
    = P(\lambda) \circ \left( \frac{\gamma_*(r)}{r}\right)^{n-\lambda+1} \circ \gamma_*
\end{equation}
for all $\gamma \in \Conf(\R^n) \subset \Conf(\R^{n+1})$. Here $\gamma_* =
(\gamma^{-1})^*$ denotes the push-forward operator on functions induced by
$\gamma$.

In the present paper, we generalize these results to a framework defined by
Riemannian metrics and discuss some applications. One of the main applications
concerns the residue families of \cite{J1}. As noted above, they can be
regarded as curved analogs of symmetry breaking differential operators. More
precisely, residue families are $1$-parameter families of conformally covariant
differential operators
\begin{equation}\label{res-f}
    D_N^{res}(h;\lambda): C^\infty(M_+) \to C^\infty(M), \; \lambda \in \C
\end{equation}
of order $N \in \N$. These are defined in the following setting. We consider a
general Riemannian manifold $(M,h)$. Let $M_+$ be an open neighborhood of
$\{0\} \times M$ in $[0,\infty) \times M$. On the open interior $M_+^\circ =
(0,\varepsilon) \times M$ of $M_+$, let $g_+ = r^{-2}(dr^2+h_r)$ be an even
Poincar\'e metric in normal form relative to $h$. Here $h_r$ is a $1$-parameter
family of metrics on $M$ with $h_0=h$. The metric $\bar{g} = dr^2+h_r$ is a
conformal compactification of $g_+$. The relevant concepts were developed in
\cite{FG} and will be recalled in Section \ref{prel}. Although the Poincar\'e
metric $g_+$ is not completely determined by the metric $h$, residue families
only depend on the Taylor coefficients of the family $h_r$ at $r=0$ which are
uniquely determined by $h$. More precisely, the even-order family
$D_{2N}^{res}(h;\lambda)$ involves $2N$ derivatives by the variable $r$ and
depends on the Taylor coefficients of $h_r$ of order $\le 2N$. The conformal
covariance of residue families describes their behavior under conformal changes
$h \to e^{2\varphi} h$ of the metric on the submanifold $M \hookrightarrow M_+
= [0,\varepsilon) \times M$.

Our generalizations of the shift operator $P(\lambda)$ are differential
operators which act on smooth functions on $M_+$ and are defined in terms of an
even Poincar\'e metric $g_+$ on $M_+^\circ$. In fact, we define a curved
version of the shift operator $P(\lambda)$ by the formula\footnote{The shift by
$2$ in the parameter $\lambda$ is a matter of conventions.}
\begin{equation}\label{shift-curved}
    S(g_+;\lambda) = r \Delta_{\bar{g}} - (2\lambda\!-\!n\!+\!1) \partial_r
    - \frac{1}{2}(\lambda\!-\!n\!+\!1) \tr(h_r^{-1}\dot{h}_r).
\end{equation}
Here the dot denotes derivatives with respect to $r$ and $\bar{g} = r^2 g_+$.
This definition can also be written in the form
\begin{equation*}\label{shift-curved-v}
   S(g_+;\lambda) = r\Delta_{\bar{g}} - (2\lambda\!-\!n\!+\!1) \partial_r
   - (\lambda\!-\!n\!+\!1) \dot{v}(r)/v(r),
\end{equation*}
where the function $v(r,\cdot) \in C^\infty(M)$ is defined by the relation
$dvol(h_r) = v(r) dvol(h)$ of volume forms. We note that, in contrast to
residue families, the shift operators are not defined only by the Taylor
coefficients of $h_r$ at $r=0$.\footnote{However, in the analytic category, the
family $h_r$ is completely determined by $h$.}

The operator $S(g_+;\lambda)$ is a second-order differential operator which
degenerates for $r=0$, i.e., on the submanifold $M$. Theorem \ref{CurvedBS}
establishes the shift property of $S(g_+;\lambda)$. This property describes its
action on functions of the form $r^{\lambda} u \in C^\infty(M_+^\circ)$, where
$u$ is an eigenfunction of the Laplacian $\Delta_{g_+}$ of the Poincar\'e
metric $g_+$ on $M_+^\circ$.

The following result describes the behavior of $S(g_+;\lambda)$ under conformal
changes of the boundary metric $h$ (Proposition \ref{ConformalCovarianceP}).

\begin{thm}\label{TA} Assume that $(M^n,h)$ is a manifold of dimension $n$.
Let $\hat{h} = e^{2\varphi}h$ be a metric in the conformal class of $h$. Let
$g_+$ be an even Poincar\'e metric in normal form relative to $h$ on
$M_+^\circ$. Let $\kappa$ be a diffeomorphism of $M_+$ which restricts to the
identity on $M$ and for which the Poincar\'e metric $\hat{g}_+ = \kappa^*
(g_+)$ is in normal form relative to $\hat{h}$. In these terms, we have
\begin{equation}\label{CCS}
   S(\hat{g}_+;\lambda) = \kappa^* \circ
   \left(\frac{\kappa_*(r)}{r}\right)^{\lambda-n} \circ S(g_+;\lambda) \circ
   \left(\frac{\kappa_*(r)}{r}\right)^{n-\lambda-1} \circ \kappa_*.
\end{equation}
\end{thm}

This result may be regarded as a version of conformal covariance. Although the
transformation law \eqref{CCS} formally resembles the equivariance property
\eqref{flat-shift-equivariance}, the former law is {\em not} a generalization
of the latter one. In fact, the diffeomorphisms $\kappa$ should not be confused
with the conformal maps $\gamma$: $\kappa$ leaves the submanifold $M$ pointwise
fixed. However, the formal similarity between the intertwining property
\eqref{flat-shift-equivariance} and the conformal transformation law
\eqref{CCS} can be explained by recognizing $P(\lambda)$ and $S(g_+;\lambda)$
(for Einstein $g_+$) both as special cases of the degenerate Laplacian $\IDD$
introduced in \cite{GW} (this concept will be recalled in Section
\ref{Laplace-deg}). Indeed, we note that
$$
   P(\lambda) = S(g_{\hyp};\lambda\!-\!2)
$$
and
\begin{equation*}
    S(g_{\hyp};\lambda) = \ID[r^2 g_{\hyp};r,\lambda\!-\!n\!+\!1] \quad \mbox{and}
    \quad S(g_+;\lambda) = - \ID[r^2 g_+;r,\lambda\!-\!n\!+\!1],
\end{equation*}
where $g_{\hyp}$ denotes the hyperbolic metric in the upper half-space and
$g_+$ is Einstein (see \eqref{PID} and \eqref{SandID}). Since $\gamma$
preserves the hyperbolic metric $g_{\hyp}$, we have
$$
   \gamma^*(r^2 g_{\hyp}) = \left(\frac{\gamma^*(r)}{r}\right)^2 (r^2 g_{\hyp})
$$
and the conformal transformation law for $\IDD$ (Proposition \ref{CTDL})
implies
\begin{align*}
    & \left( \frac{r}{\gamma^*(r)}\right)^{\lambda-n-2} \circ \gamma^* \circ P(\lambda)
    \circ \gamma_* \circ \left( \frac{r}{\gamma^*(r)}\right)^{n-\lambda+1} \notag \\
    & = \left( \frac{r}{\gamma^*(r)}\right)^{\lambda-n-2} \circ \gamma^* \circ \ID[r^2
    g_{\hyp};r;\lambda\!-\!n\!-\!1] \circ \gamma^* \circ \left(
    \frac{r}{\gamma^*(r)}\right)^{n-\lambda+1} \notag \\
    & = \left( \frac{r}{\gamma^*(r)}\right)^{\lambda-n-2} \circ \gamma^* \circ \ID[\gamma^*(r^2
    g_{\hyp});\gamma^*(r);\lambda\!-\!n\!-\!1] \circ \gamma^* \circ \left(
    \frac{r}{\gamma^*(r)}\right)^{n-\lambda+1} \notag \\
    & = \ID[r^2 g_{\hyp};r,\lambda\!-\!n\!-\!1] = P(\lambda).
\end{align*}
This proves \eqref{flat-shift-equivariance}. A similar calculation gives
\eqref{CCS} (Remark \ref{ConformalTrafoP}).

For $N \in \N$, we define the compositions
\begin{equation}\label{shift-c-op}
    S_N(g_+;\lambda) \st \underbrace{S(g_+;\lambda) \circ \cdots
    \circ S(g_+;\lambda\!+\!N\!-\!1)}_{N \, factors}.
\end{equation}
We shall refer to these operators as {\em iterated shift operators} or simply
also as shift operators. Theorem \ref{TA} implies that all iterated shift
operators $S_N(g_+;\lambda)$ are conformally covariant (in the sense as in
\eqref{CCS}). The following result states that residue families \eqref{res-f}
can be written in terms of iterated shift operators (Corollary \ref{RFvsSF}).
Its proof rests on the shift property of the shift operators. Let the embedding
$\iota: M \hookrightarrow M_+$ be defined by $m \mapsto(0,m)$.

\begin{thm}\label{TB} Assume that $(M^n,h)$ is a Riemannian manifold of dimension $n$.
Let $N \in \N$ so that $2N \le n$ if $n$ is even. Then the residue family
$D_{2N}^{res}(h;\lambda)$ of order $2N$ is proportional to the composition of
the family $\lambda \mapsto S_{2N}(g_+;\lambda+n-2N)$ with the restriction
$\iota^*$ to $M$. More precisely, we have
\begin{equation}\label{Dresshift}
   (-2N)_N \left(\lambda\!+\!\frac{n\!+\!1}{2}\!-\!2N\right)_N D_{2N}^{res}(h;\lambda)
   = \iota^* S_{2N}(g_+;\lambda\!+\!n\!-\!2N).
\end{equation}
A similar formula holds true for odd-order residue families.
\end{thm}

Some comments on this result are in order.

By construction, the family $S_{2N}(g_+;\lambda)$ involves $4N$ derivatives in
the variable $r$ and depends on all Taylor coefficients of $h_r$. The identity
\eqref{Dresshift} shows that its composition with the restriction operator
$\iota^*$ actually involves only $2N$ derivatives in $r$ and depends only on
the Taylor coefficients of $h_r$ up to order $2N$. In particular, the
compositions $\iota^* S_{2N}(g_+;\lambda)$ are completely determined by $h$. In
the following, we shall denote these compositions by $\S_{2N}(h;\lambda)$. For
$N\in \N$ not satisfying the assumptions in Theorem \ref{TB}, the compositions
$\iota^* S_{2N}(g_+;\lambda)$ in general are not determined only by $h$.

The product formula \eqref{Dresshift} yields a new expression for the residue
families. It extends a result of \cite{C,FOS} in the flat case. The description
of residue families in terms of restrictions to $M$ of ''powers'' of a
universal shift operator living in a neighborhood of $M$ resembles the
construction of the conformally covariant powers of the Laplacian (GJMS
operators) of $h$ by powers of the Laplacian of an ambient metric associated to
$h$ \cite{GJMS}.

Theorem \ref{TB} can be used to deduce properties of residue families from
properties of shift operators and vice versa. In particular, the conformal
covariance of $S_N(g_+;\lambda)$ implies a conformal covariance law for residue
families. This reproves \cite[Theorem 6.6.3]{J1}.

In addition, Theorem \ref{TB} enables us to give easy proofs of the systems of
factorization identities of residue families which play an important role in
\cite{J1,J2} in connection with the description of recursive structures for
GJMS operators and $Q$-curvatures. Our new proofs of these factorization
identities rest on two basic facts. The first one (Theorem \ref{BigGJMS}) is
also of independent interest.

\begin{thm}\label{TC} Assume that $N \in \N$ with $2N \le n$ if $n$ is even.
Then
\begin{equation}\label{Ident-2}
   S_N \left(g_+;\frac{n-1}{2}\right) = r^N P_{2N}(\bar{g}),
\end{equation}
up to an error term in $O(r^\infty)$ for $n$ odd and $o(r^{n-N})$ for $n$ even.
Moreover, the equality holds true without an error term if $g_+$ is Einstein.
\end{thm}

Here $P_{2N}(\bar{g})$ is a GJMS operator of the conformal compactification
$\bar{g}$ of the Poincar\'e metric $g_+$ in normal form relative to
$h$.\footnote{For odd $n$, the operators $P_{2N}(\bar{g})$ are well-defined for
all $N \in \N$. See the comments at the beginning of Section \ref{recover}.}
The second basic fact is the identity
\begin{equation}\label{Ident-1}
   \S_{2N}\left(h;\frac{n}{2}\!-\!N\right) = ((2N\!-\!1)!!)^2 P_{2N}(h) \iota^*
\end{equation}
(Theorem \ref{BSOperatorVsGJMS}). This formula reproves a special case of a
result of \cite{GW}. Together with
\begin{equation}\label{Ident-3}
   \S_{2N}\left(h;\frac{n\!-\!1}{2}\!-\!N\right) = (2N)! \iota^* P_{2N}(\bar{g})
\end{equation}
it shows that the operators $\S_{2N}(h;\lambda)$ interpolate between GJMS
operators for the metrics $h$ and $\bar{g}$.

The coefficients of the families $S_N(g_+;\lambda)$ depend on the parameters
$r$ and $\lambda$. A closer study of both dependencies seems to be of interest.
Theorem \ref{TC} may be regarded as a result in that direction. More results in
this direction are discussed in Section \ref{expansions}.

Finally, through the relation between residue families and iterated shift
operators, we derive a new formula for the critical $Q$-curvature $Q_n(h)$ of a
manifold $(M^n,h)$ of even dimension $n$ (Theorem \ref{Q-holo-new}).

\begin{thm}\label{TD} Let $n$ be even. Then
\begin{equation}\label{Q-S-holo}
   Q_n(h) = c_n \S_{n-1}(h;0) \partial_r (\log v),
\end{equation}
where $c_n = (-1)^{\frac{n}{2}} 2^{n-2} (\Gamma(\frac{n}{2})/\Gamma(n))^2$.
\end{thm}

There is an interesting formal resemblance of the latter formula for the
critical $Q$-curvature with a formula of Fefferman and Hirachi \cite{FH}.

Theorem \ref{TD} extends to all subcritical $Q$-curvatures $Q_{2N}(h)$ for $2N
< n$ in the form
\begin{equation}\label{Q-S-holo-g}
    Q_{2N}(h) = c_{2N} \S_{2N-1}\left(h;\frac{n}{2}-N\right) \partial_r (\log v),
\end{equation}
where $c_{2N} = (-1)^{N} 2^{2N-2} (\Gamma(N)/\Gamma(2N))^2$ (Theorem
\ref{Q-holo-g}).

Combining \eqref{Q-S-holo} and \eqref{Q-S-holo-g} with Theorem \ref{TB}, yields
formulas for $Q$-curvatures in terms of residue families. These turn out to be
equivalent to the holographic formulas proved in \cite{GJ,J-holo}. In other
words, these holographic formulas for $Q$-curvatures can be viewed as natural
consequences of Theorem \ref{TB}.

The operator $S(g_+;\lambda)$ appeared in the literature in different contexts.
The construction of asymptotic expansions of eigenfunctions for the Laplacian
$\Delta_{g_+}$ of a Poincar\'e metric in \cite{GZ} involved a second-order
operator $\D_s$. In \cite{GW}, the operator $\D_s$ was interpreted and
generalized within tractor calculus. This led to the definition of the
so-called degenerate Laplacian $\IDD$ which played a role in the discussion
after Theorem \ref{TA}. Compositions as in \eqref{shift-c-op} of these
operators were used in \cite{GW} in connection with the construction of
asymptotic expansions in a more general eigenfunction problem. The relations
among these construction will be described in Section \ref{prel}. In \cite{C},
Clerc gave a representation theoretical alternative construction of a family of
symmetry breaking differential operators introduced in \cite{J1} in terms of
compositions of shifted operators $P(\lambda)$. Theorem \ref{TB} is a
generalization of his result to the curved setting.

The paper is organized as follows. After a collection of background material,
we use Section \ref{BSOperator} to introduce the shift operator
$S(g_+;\lambda)$ and prove basic properties. In Section \ref{OnJuhlsFormulae},
we establish the connection between residue families and iterated shift
operators. Section \ref{applications} is devoted to various applications. Here
we provide easy new proofs of the factorization identities of residue families
and discuss holographic formulas for $Q$-curvatures. In Section \ref{panorama},
we illustrate the main results in low-order cases. In the final section, we
speculate on the role of iterated shift operators in the theory of the building
block operators $\M_{2N}$ \cite{J2,J3} of GJMS operators. In this connection,
we derive a new formula for the so-called holographic Laplacian \cite{J3} for
the metric $\bar{g}$.

{\em Acknowledgment.} The first two authors are grateful to {\AA}rhus University
for hospitality, stimulating atmosphere and financial support.

\section{Preliminaries}\label{prel}

In the present section, we fix notation and describe the general setting. We
recall basic facts on GJMS operators, $Q$-curvatures, residue families, shift
operators and the degenerate Laplacian.

\subsection{General notation}\label{notation}

$\N$ is the set of natural numbers and $\N_0$ the set of non-negative integers.
For a complex number $a\in\C$ and an integer $N\in\N$, the Pochhammer symbol
$(a)_N$ is defined by $(a)_N \st a(a+1) \cdots (a+N-1)$. We also set $(a)_0 \st
1$. $C^\infty(M)$ is the space of smooth functions on the manifold $M$ and
$C_c^\infty(M)$ denotes the subspace of functions with compact support.
$\Delta_g$ denotes the Laplacian of a Riemannian metric $g$ on a manifold $M$
acting on $C^\infty(M)$. Here we use the convention that $-\Delta_g$ is
non-negative, i.e., $-\Delta_g = \delta_g d$, where $\delta_g$ is the formal
adjoint of the differential $d$. $\Ric(g)$ and $\tau(g)$ denote the Ricci
tensor and the scalar curvature of $g$. On a manifold $(M^n,g)$ of dimension
$n$, we set $\J(g) = \frac{1}{2(n-1)} \tau(g)$ and define the Schouten tensor
of $g$ by $\Rho(g) = \frac{1}{n-2} (\Ric(g) - \J(g) g)$. We shall also write
simply $\Rho$ and $\J$ if the metric is clear by context. The symbol $\circ$
denotes compositions of operators.

\subsection{GJMS operators and $Q$-curvatures}\label{setting}

Let $(M^n,h)$ be a Riemannian manifold of dimension $n \geq 3$. For $N \in \N$
if $n$ is odd and $\N \ni N \leq \frac{n}{2}$ if $n$ is even, the GJMS
operators are conformally covariant differential operators
\begin{equation*}
    P_{2N}(h): C^\infty(M)\to C^\infty(M)
\end{equation*}
of order $2N$ which are of the form $P_{2N}(h) = \Delta_h^N + LOT$, where $LOT$
denotes lower-order terms. These lower-order terms only depend on covariant
derivatives of the curvature of $h$. Under conformal changes
$\hat{h}=e^{2\varphi}h$ with $\varphi \in C^\infty(M)$ of the metric, the GJMS
operators satisfy
\begin{equation}\label{eq:ConformalPropertyGJMS}
    P_{2N}(\hat{h}) = e^{-(\frac{n}{2}+N)\varphi} \circ P_{2N}(h) \circ e^{(\frac{n}{2}-N)\varphi}.
\end{equation}
In \cite{GJMS}, these operators were constructed in terms of powers of the
Laplacian of an ambient metric associated to $h$.

The GJMS operators generalize the well-known Yamabe operator
\begin{equation}\label{Yamabe}
    P_2 = \Delta  - \left(\frac{n}{2}-1\right) \J
\end{equation}
and the Paneitz operator
\begin{equation}\label{Paneitz}
    P_4 = \Delta^2 + \delta \left((n\!-\!2)\J h - 4\Rho\right) \# d
    + \left(\frac{n}{2}\!-\!2\right)
    \left(\frac{n}{2}\J^2 \!-\! 2 |\Rho|^2 \!-\! \Delta \J \right),
\end{equation}
where $|\Rho|^2 = \Rho_{ij} \Rho^{ij}$ and $\#$ indicates the natural action of
symmetric $2$-tensors on $\Omega^1(M)$.

In odd dimension $n$, we have GJMS operators $P_{2N}$ of any order $2N$, $N \in
\N$. But, for general metrics $h$ in even dimension $n$, the restriction $2N
\le n$ is necessary both for the definition of $P_{2N}(h)$ and for the
existence of conformally covariant differential operators with leading term
$\Delta_h^N$ \cite{non-ex,GH}.

Explicit formulas for GJMS operators for general metrics are very complicated
\cite{J2}. But for some special metrics, they may be given by closed formulas.
In particular, for Einstein manifolds $(M^n,h)$ they are given by the formula
\begin{equation}\label{eq:GJMSOnEinstein}
    P_{2N}(h) = \prod_{l=1}^N \left(\Delta_h - 2\mu
    \left(\frac{n}{2}+l-1\right)\left(\frac{n}{2}-l\right)\right)
\end{equation}
for all $N\in\N$, where the constant $\mu\in\R$ is defined by
$\Ric(h)=2\mu(n-1)h$ \cite{FG}. Here the above restriction on their order is
irrelevant.

It is a basic observation \cite{sharp} that
\begin{equation}\label{Q-def}
   P_{2N}(h)(1) = (-1)^N \left(\frac{n}{2}-N\right) Q_{2N}(h)
\end{equation}
for a scalar curvature invariant $Q_{2N}(h) \in C^\infty(M^n)$ of order $2N$ .
The quantities $Q_{2N}(h)$ are well-defined by \eqref{Q-def} as long as $2N <
n$. These curvature quantities are called the subcritical $Q$-curvatures. Their
analogs of even order $n$ can be defined by analytic continuation in dimension
$n$ through the subcritical $Q$-curvatures. The quantity $Q_n(h)$ is called the
critical $Q$-curvature of $(M^n,h)$. Under the respective conditions $n
> 2$ and $n > 4$, \eqref{Yamabe} and \eqref{Paneitz} yield the subcritical $Q$-curvatures
\begin{equation}\label{sub-Q}
   Q_2 = \J \quad \mbox{and} \quad Q_4 = \frac{n}{2}\J^2 \!-\! 2 |\Rho|^2 \!-\! \Delta \J
\end{equation}
of order $2$ and $4$. Here we suppress the obvious dependence
of constructions on the metric $h$. By continuation in dimension $n$, we define
the respective critical $Q$-curvatures
\begin{equation}\label{Q24}
   Q_2 = \J \quad \mbox{and} \quad Q_4 = 2 \J^2 \!-\! 2 |\Rho|^2 \!-\! \Delta \J
\end{equation}
in dimension $n=2$ and $n=4$.

In the following, we shall often simplify notation by omitting the composition
sign $\circ$ in compositions with multiplication operators. It also will often
lead to simplifications to suppress the obvious dependence of constructions on
$h$.

\subsection{Poincar\'e metrics, eigenfunction expansions and GJMS operators}
\label{PEandGJMS}

In the present section, we briefly recall basic definitions concerning
Poincar\'e metrics in the sense of Fefferman and Graham \cite{FG} and recall a
description of GJMS operators of $(M,h)$ in terms of eigenfunctions of the
Laplacian of an associated Poincar\'e metric on $M_+^\circ$ \cite{GZ}. This
description will be of central importance for all later constructions.

Let $M$ be a manifold of dimension $n \ge 3$. Let $M_+$ be an open neighborhood
of $\{0\} \times M$ in $[0,\infty) \times M$, i.e., $M_+ = [0,\varepsilon)
\times M$ for some $\varepsilon > 0$. We use the coordinate $r$ on the first
factor. We define the embedding $\iota: M \to M_+$ by $\iota(m) = (0,m)$. Let
$M_+^\circ = (0,\varepsilon) \times M$. A smooth metric
\begin{equation}\label{normal}
    g_+ = r^{-2} (dr^2 + h_r)
\end{equation}
on $M_+^\circ$ is called a Poincar\'e metric in normal form relative to a
metric $h$ on $M$ if $\bar{g} = r^2 g_+$ extends to $M_+$, $\bar{g}$ restricts
to $h$, i.e., $\iota^*(\bar{g}) = h$, and the Ricci tensor of $g_+$ satisfies
the Einstein condition
\begin{equation}\label{Ricci-odd}
    \Ric(g_+) + n g_+ = O(r^\infty)
\end{equation}
for odd $n \ge 3$ and the Einstein condition
\begin{equation}\label{Ricci-even}
   \Ric(g_+) + ng_+ = O(r^{n-2})
\end{equation}
together with the vanishing trace condition
\begin{equation}\label{VT}
   \tr_h (\iota^* (r^{-n+2} (\Ric(g_+) + ng_+))) = 0
\end{equation}
for even $n \ge 4$. The metric $\bar{g} = r^2 g_+$ on $M_+$ is called a
conformal compactification of $g_+$. The family $h_r$ in \eqref{normal} is a
smooth $1$-parameter family of metrics on $M$.

If, for odd $n$, we also assume that $h_r$ has an {\em even} expansion
\begin{equation}\label{odd}
    h_r = h_0 + r^2 h_2 + r^4 h_4 + \cdots,
\end{equation}
then the condition \eqref{Ricci-odd} implies that the coefficients
$h_2,h_4,\dots$ are uniquely determined by $h_0=h$. These metrics are
conveniently referred to as {\em even} Poincar\'e metrics. One may consider the
conformal compactification $\bar{g}$ of an even Poincar\'e metric as a smooth
metric on the larger space $(-\varepsilon,\varepsilon) \times M$. In the real
analytic category, $h_r$ converges and $\Ric(g_+) + ng_+ = 0$ in a neighborhood
of $\{0\} \times M$.

For even $n$, the situation is more complicated. In that case, the family $h_r$
has an expansion of the form
\begin{equation}\label{even}
   h_r = \underbrace{h_0 + r^2 h_2 + \cdots + r^{n-2} h_{n-2}}_{even \; powers}
   + r^n (h_n + \log r h_0^{(1)}) + \cdots.
\end{equation}
The condition \eqref{Ricci-even} uniquely determines the coefficients
$h_2,\dots,h_{n-2}$ by $h_0=h$. Moreover, the vanishing trace condition
\eqref{VT} can be satisfied and determines the $h$-trace of $h_n$. However, the
trace-free part of $h_n$ is {\em not} determined by $h$.

In general, the higher-order solutions of the Einstein condition contain $\log
r$ terms. The first $\log r$ coefficient $h_0^{(1)}$ (Fefferman-Graham
obstruction tensor) is uniquely determined by $h$ and trace-free. For specific
choices of the trace-free part of $h_n$, the condition
\begin{equation*}
   \Ric(g_+) + ng_+ = O(r^\infty)
\end{equation*}
may be satisfied by solutions with expansions of the form
\begin{equation}\label{even-general}
   h_r = \underbrace{h_0 + r^2 h_2 + \cdots + r^n h_n + \cdots}_{even \; powers}
   + \sum_{j=1}^\infty (r^n \log r)^j h_r^{(j)}
\end{equation}
with even families $h_r^{(j)}$. The $\log r$ terms in these expansion vanish
iff the obstruction tensor vanishes.

Now assume that $g_+ = r^{-2} (dr^2 + h_r)$ is a Poincar\'e metric in normal
form relative to $h$. Hence $h_0 = h$. Let $\hat{h} = e^{2\varphi} h$ be a
metric in the same conformal class as $h$. Then a suitable change of
coordinates brings $g_+$ into normal form relative to $\hat{h}$. Following
\cite[Section 5]{GL} and \cite[Proposition 4.3]{FG}, we briefly recall the
arguments proving this basic observation. The metric $g_+$ is asymptotically
hyperbolic since $|dr/r|_{g_+} = 1$ on $r=0$. The latter property suffices to
prove the existence of $u \in C^\infty(M_+)$ so that for $\rho = r e^u \in
C^\infty(M_+)$ we have
$$
   |d \rho|^2_{\rho^2 g_+} = 1 \; \mbox{near $M$}.
$$
Here the restriction of $u$ to $r=0$ can be arbitrarily chosen. Now let
$\mathfrak{X} = \grad_{\rho^2 g_+}(\rho)$ be the gradient field of $\rho$ with
respect to the conformal compactification $\rho^2 g_+$ of $g_+$. Let
$\Phi^t_\mathfrak{X}$ be the flow of $\mathfrak{X}$. In these terms, we define
the map
$$
   \kappa: [0,\varepsilon) \times M \ni (\lambda,x) \mapsto (\Phi^\lambda_\mathfrak{X})(x)
   \in [0,\varepsilon) \times M
$$
for sufficiently small $\varepsilon$. Then $\kappa(0,x) = x$ and
$\kappa^*(\rho)(\lambda,x) = \lambda$. The gradient field $\mathfrak{X}$ is
orthogonal to the slices $\rho^{-1}(\lambda)$. It follows that
$$
    \kappa^*(\rho^2 g_+) = d\lambda^2 + k_\lambda
$$
for some some $1$-parameter family $k_\lambda$. Hence
$$
   \kappa^*(g_+) = \frac{1}{\kappa^*(\rho)}  \kappa^* (\rho^2 g_+) =
   \lambda^{-2} (d\lambda^2 + k_\lambda).
$$
Finally, we note that
$$
   k_0 = \iota^* (d\lambda^2 + k_\lambda) = \iota^* \kappa^* (\rho^2 g_+)
   = (\kappa \iota)^* (\rho^2 g_+) = \iota^* (\rho^2 g_+)
   = \iota^* \left(\frac{\rho}{r}\right)^2 h_0 = e^{2\iota^*(u)} h_0
$$
(with obvious embeddings $\iota$). In other words, for the choice
$\iota^*(u)=\varphi$, $\kappa^*(g_+)$ is a Poincar\'e metric in normal form
relative to $\hat{h}$.


The volume function $v(r) \in C^\infty(M^n)$ is defined by the relation
\begin{equation}\label{eq:VolumeFunction}
    dvol(h_r) = v(r) dvol(h)
\end{equation}
of volume forms. For odd $n$ and even Poincar\'e metrics, the function $v(r)$
has an even Taylor series $v(r) = 1 + r^2 v_2 + r^4 v_4 + \cdots$. Similarly,
for even $n$, we have
$$
   v(r) = 1 + r^2 v_2 + \cdots + r^n v_n + \dots.
$$
The indicated coefficients in these expansions are locally determined by $h$.
The coefficients $v_{2j}$ are called the renormalized volume coefficients of
$(M^n,h)$ \cite{G-vol}. Note that
\begin{equation}\label{v-trace}
   \dot{v}(r)/v(r) = \frac{1}{2} \tr(h_r^{-1}\dot{h}_r).
\end{equation}
In particular, for even $n$, the coefficient $v_n$ only depends on the trace of
$h_n$. We also set
\begin{equation}\label{eq:WFunction}
    w(r) \st \sqrt{v(r)}.
\end{equation}
Then $w(r) = \sum_{j \ge 0} r^{2j} w_{2j}$. The coefficients $w_{2N}$ are
polynomials in $v_{2k}$ for $k \le N$.

A routine calculation shows that the Laplacian of $g_+$ takes the form
\begin{equation}\label{eq:PELaplacian}
    \Delta_{g_+} = r^2 \Delta_{h_r} + r^2 \partial_r^2 - (n\!-\!1)r\partial_r +
    \frac{1}{2} \tr(h_r^{-1}\dot{h}_r) r^2 \partial_r,
\end{equation}
where $\Delta_{h_r}$ is the Laplacian of $h_r$ and $\dot{h}_r$ denotes the
derivative of $h_r$ with respect to $r$. Since $\bar{g}$ and $g_+$ are
conformally equivalent, another calculation shows that
\begin{equation}\label{eq:CCLaplacian}
   \Delta_{\bar{g}} = r^{-2}(\Delta_{g_+} + (n\!-\!1)r\partial_r)
   = \Delta_{h_r}+\partial_r^2 + \frac{1}{2} \tr(h_r^{-1}\dot{h}_r) \partial_r.
\end{equation}
For even $h_r$, this formula is well-defined on $(-\varepsilon,\varepsilon)
\times M$.

We continue with the discussion of GJMS operators of $(M^n,h)$. These operators
can be described in terms of asymptotic expansions of the solutions of the
equation
$$
   \Delta_{g_+} u + \lambda(n-\lambda) u = 0, \; \lambda \in \C.
$$

In the case of the hyperbolic ball such solutions can be represented as
Helgason-Poisson transforms of distributions (or even hyperfunctions) on the
boundary $S^n$ \cite{Hel}. Here we consider solutions with {\em smooth}
boundary values in the curved setting. A corresponding Poisson transform was
constructed in \cite[Proposition 3.5]{GZ}. Its definition rests on a local
(near the boundary) asymptotic analysis of the expansions of eigenfunctions and
global mapping properties of the resolvent of $\Delta_{g_+}$ acting on
appropriate functional spaces.

More precisely, we consider eigenfunctions $u$ with asymptotic expansions of
the form
\begin{equation}\label{ansatz}
    u(r,x) \sim \sum_{j \geq 0} r^{\lambda+2j} a_{2j}(\lambda)(x)
    + \sum_{j \geq 0} r^{n-\lambda+2j} b_{2j}(\lambda)(x)
\end{equation}
with coefficients $a_{2j}(\lambda), b_{2j}(\lambda) \in C^\infty(M)$. The
coefficients in both sums in \eqref{ansatz} are determined by the respective
leading coefficients $a_0(\lambda)$ and $b_0(\lambda)$ through a recursive
algorithm. Moreover, for a global eigenfunction $u$, both leading coefficients
are related by a scattering operator $\Sc(\lambda)$.

We recall these constructions in some more detail. First, a local asymptotic
analysis yields a map
$$
   \Phi(\lambda): C^\infty(M) \to r^{n-\lambda} C^\infty(M_+), \quad \Re(\lambda)
   > n/2
$$
so that
$$
    (\Delta_{g_+} + \lambda(n-\lambda)) \Phi(\lambda) f = O(r^\infty).
$$
For even Poincar\'e metrics, it has the form
$$
   \Phi(\lambda) f = r^{n-\lambda} f + \sum_{j \ge 1} r^{n-\lambda+2j} \T_{2j}(n-\lambda) f,
$$
where $\T_{2j}(\lambda)$ are meromorphic families of differential operators on
$M$ of order $2j$. Next, the resolvent $R(\lambda) = (\Delta_{g_+} +
\lambda(n-\lambda))^{-1}: L^2(M_+) \to L^2(M_+)$ is holomorphic for
$\Re(\lambda)>n$. Moreover, its restriction to the space of smooth functions
which vanish of infinite order on the boundary, admits a meromorphic
continuation to $\C$. The range of that restriction of $R(\lambda)$ is
contained in $r^\lambda C^\infty(M_+)$. Then the family
\begin{equation}\label{PT-def}
   \P(\lambda) \st \Phi(\lambda) - R(\lambda) (\Delta_{g_+} + \lambda(n-\lambda)) \Phi(\lambda)
\end{equation}
of {\em Poisson transforms} is meromorphic on $\Re(\lambda) > n/2$ with poles
only for real $\lambda$ with $\lambda(n-\lambda) \in \sigma_d(-\Delta_{g_+})
\subset (0,(n/2)^2)$.\footnote{The poles of $\Phi(\lambda)$ in $n/2 + \N$
cancel in the sum \eqref{PT-def}.} $\P(\lambda)$ is continuous up to
$\Re(\lambda) =n/2$, $\lambda \ne n/2$. It satisfies
$$
   (\Delta_{g_+} + \lambda(n-\lambda)) \P(\lambda) = 0.
$$
Away from the real poles in $\Re(\lambda) > n/2$ and $\lambda \not\in n/2 +
\N_0$, we have
\begin{equation}\label{PT-AS}
   \P(\lambda) f = r^\lambda G + r^{n-\lambda} F
\end{equation}
for $F, G \in C^\infty(M_+)$ with $\iota^*(F) = f$.\footnote{For $\lambda \in
n/2+\N$ the corresponding expansion contains a $\log r$-term.} $f$ is viewed as
the {\em boundary value} of the eigenfunction $u = \P(\lambda) f$. For the
details see \cite[Proposition 3.5]{GZ}. Later we shall use \eqref{PT-AS} for
$\Re(\lambda)=n/2$, $\lambda \ne n/2$.

The scattering operator $\Sc(\lambda): C^\infty(M) \to C^\infty(M)$ is defined
by
$$
    \Sc(\lambda): f \mapsto \iota^*(G)
$$
for $G$ as in \eqref{PT-AS}. It follows that in the asymptotic expansion
\eqref{ansatz} of $u=\P(\lambda)f$ the coefficients are given by
$$
    b_{2j}(\lambda) = \T_{2j}(n-\lambda) f \quad \mbox{and} \quad
    a_{2j}(\lambda) = \T_{2j}(\lambda) \Sc(\lambda) f.
$$
Note that $\T_0(\lambda) = \id$.

The families $\T_{2j}(\lambda)$ only depend on the Taylor series of $h_r$. For
odd $n$ and even Poincar\'e metrics, these are determined by $h$. Therefore, we
write $\T_{2j}(\lambda) = \T_{2j}(h;\lambda)$. For even $n$, only the families
$\T_{2j}(\lambda)$ with $2j \le n$ are determined by $h$ and we indicate that
dependence accordingly. However, the scattering operator $\Sc(\lambda)$ is a
global object which depends on the chosen metric on $M_+^\circ$.

The families $\T_{2j}(h;\lambda)$ are meromorphic in $\lambda$, with simple
poles at $\lambda=\frac{n}{2}-k$ for $k=1,\ldots,j$. Under the restriction $2j
\le n$ for even n, the residue of $\T_{2j}(h;\lambda)$ at
$\lambda=\frac{n}{2}-j$ is proportional to the GJMS operator $P_{2j}(h)$ on
$(M,h)$. More precisely, we have the basic residue formula
\begin{equation}\label{Res-SO}
   \Res_{\lambda = \frac{n}{2}-j} (\T_{2j}(h;\lambda)) = \frac{1}{2^{2j}j!(j-1)!} P_{2j}(h)
\end{equation}
describing GJMS operators in terms of asymptotic expansions of eigenfunctions
of $\Delta_{g_+}$. In \cite[Section 4]{GZ} this formula is derived from the
original ambient metric definition of the GJMS operators.

The residue at $\lambda=\frac{n}{2}+N$ of the right-hand side of the expansion
\eqref{ansatz} yields the contribution
$$
    r^{\frac{n}{2}+N} \left(\Res_{\frac{n}{2}+N} (\Sc(\lambda)) +
    \Res_{\frac{n}{2}-N}(\T_{2N}(\lambda))\right).
$$
Under mild assumptions, this residue vanishes. Hence \eqref{Res-SO} implies the
residue formula
\begin{equation}\label{Res-scattering}
   \Res_{\lambda = \frac{n}{2}+j} (\Sc(\lambda)) = -\frac{1}{2^{2j}j!(j-1)!} P_{2j}(h)
\end{equation}
for the poles of the scattering operator \cite[Theorem 1]{GZ}.

We emphasize that the formula \eqref{Res-SO} only rests on the local analysis
of eigenfunctions near the boundary $r=0$. However, the definition of the
scattering operator and the construction of {\em exact} eigenfunctions involves
the {\em global} resolvent $R(\lambda)$ on an asymptotically hyperbolic
manifold $M_+^\circ$. For a given closed $M$, a simple choice for $M_+^\circ$
is $M_+^\circ = (0,1) \times M$. In that case, the metric on $M_+^\circ$ is a
Poincar\'e metric near {\em both} boundary components at $r=0$ and $r=1$. The
resulting scattering operator then acts on smooth functions on the disjoint
union of both copies of $M$.

A simple special case of the latter situation is the scattering operator of the
hyperbolic cylinder $M_+^\circ = \Gamma \backslash \HS^{n+1}$ by a cocompact
discrete subgroup of $SO(1,n)^\circ$ regarded as a subgroup of
$SO(1,n+1)^\circ$ (using a trivial embedding). Then $M_+^\circ$ can be
identified with a cylinder $(-\infty,\infty) \times M$ with compact
cross-section $M = \Gamma \backslash \HS^n$. The boundary consists of two
copies of $M$. The scattering operator of the cylinder acts on $C^\infty(M)
\oplus C^\infty(M)$. It decomposes into the direct sum of endomorphisms on the
spaces $E(\mu) \oplus E(\mu)$ generated by the eigenspaces $E(\mu) = \{u \in
C^\infty(M) | -\Delta u = \mu(n\!-\!1-\!\mu)u \}$, where $\Delta$ is the
Laplacian of the hyperbolic metric on $M$. The restriction $\Sc(\lambda;\mu)$
of $\Sc(\lambda)$ to this space is given by \cite[Appendix B]{PP}
\begin{equation*}
   \Sc(\lambda;\mu) = 2^{n-2\lambda} \frac{1}{\pi} \frac{\Gamma(\frac{n}{2}\!-\!\lambda)}
   {\Gamma(\lambda\!-\!\frac{n}{2})} \Gamma(\lambda\!-\!\mu) \Gamma(\lambda\!-\!(n\!-\!1\!-\!\mu))
   \begin{pmatrix}
   \sin \pi(\frac{n}{2}-\mu) & \sin \pi(\frac{n}{2}-\lambda) \\
   \sin \pi(\frac{n}{2}-\lambda) & \sin \pi(\frac{n}{2}-\mu)
   \end{pmatrix}.
\end{equation*}
Although $\Sc(\lambda;\mu)$ contains off-diagonal terms, its residues at
$\frac{n}{2}+N$ are diagonal. More precisely, we find
\begin{equation*}
   \Res_{\lambda=\frac{n}{2}+N}(\Sc(\lambda;\mu))
   = -\frac{1}{2^{2N} N!(N\!-\!1)!} \prod_{j=\frac{n}{2}}^{\frac{n}{2}+N-1}
   (-\mu(n\!-\!1\!-\!\mu) + j(n\!-\!1\!-\!j)) \id.
\end{equation*}
This result implies the residue formula
\begin{align*}
   \Res_{\lambda=\frac{n}{2}+N}(\Sc(\lambda))
   & = -\frac{1}{2^{2N} N!(N\!-\!1)!} \prod_{j=\frac{n}{2}}^{\frac{n}{2}+N-1}
   (\Delta_M \!+\! j(n\!-\!1\!-\!j)) \\
   & = -\frac{1}{2^{2N} N!(N\!-\!1)!} P_{2N}(M,g_{\hyp})
\end{align*}
which confirms the residue formula \eqref{Res-scattering} of \cite{GZ}.

\subsection{Residue families}\label{res-fam}

We recall the concept of residue families introduced in \cite{J1}. We assume
that $g_+$ is an even Poincar\'e metric relative to $h$. For $N \in \N_0$ with
$2N \le n$ for even $n$, we define a polynomial $1$-parameter family of
differential operators $C^\infty(M_+) \to C^\infty(M)$ through
\begin{equation}\label{res-fam-even}
    D_{2N}^{res}(h;\nu) \st 2^{2N}N! \left(-\frac{n}{2}\!-\!\nu\!+\!N\right)_N
    \delta_{2N}(h;\nu\!+\!n\!-\!2N)
\end{equation}
and
\begin{equation}\label{res-fam-odd}
    D_{2N+1}^{res}(h;\nu) \st 2^{2N}N! \left(-\frac{n}{2}\!-\!\nu\!+\!N\!+\!1\right)_N
    \delta_{2N+1}(h;\nu\!+\!n\!-\!2N\!-\!1),
\end{equation}
where the family $\delta_N(h,\nu): C^\infty(M_+) \to C^\infty(M)$ is defined by
the residue formula
\begin{equation}\label{eq:DefDeltaN}
    \Res_{\lambda=-\nu-1-N} \left(\int_{M_+} r^{\lambda} u \varphi dvol(\bar{g}) \right)
    = \int_M f \delta_N(h;\nu) \varphi dvol(h).
\end{equation}
Here $u$ is an eigenfunction of $-\Delta_{g_+}$ on $M_+^\circ$ with eigenvalue
$\nu (n-\nu)$ and boundary value $f$ (Section \ref{PEandGJMS}), and we use test
functions $\varphi \in C_c^\infty(M_+)$ (\cite[(6.6.11)]{J1}). Note that
$D_0^{res}(\lambda) = \iota^*$. The above definitions yield the formula
\begin{equation}\label{deltaN}
    \delta_N(h;\nu)=\sum_{j=0}^N \frac{1}{(N-j)!}[\T_{j}^*(h;\nu) v_0 + \cdots
    + \T_0^*(h;\nu) v_j] \iota^* \partial_r^{N-j}
\end{equation}
(\cite[Definition 6.6.2]{J1}) in terms of {\em solution operators}
$\T_{2j}(h;\nu)$ and renormalized volume coefficients $v_{2j}$. Here the
operator $\T_{j}^*(h;\nu)$ denotes the formal adjoint of $\T_{j}(h;\nu)$ with
respect to the scalar product on $C^\infty(M)$ defined by $h$. Since $g_+$ is
assumed to be even, solution operators and renormalized volume coefficients
with odd indices vanish. Although the residue families $D_N^{res}(h;\nu)$ are
defined in terms of asymptotic expansions of eigenfunctions of $\Delta_{g_+}$,
the assumptions on $N$ guarantee that they only depend on the Taylor
coefficients of $h_r$ which are determined by $h$. This justifies the notation.
For even $n$, the family $D_n^{res}(h;\nu)$ sometimes is called the critical
residue family. The formula \eqref{deltaN} shows that, for even $n$, also the
odd-order residue family $D_{n+1}^{res}(h;\nu)$ is determined by $h$.

In later sections, we shall prefer to use the notation $D_N^{res}(h;\lambda)$
instead of $D_N^{res}(h;\nu)$.

\subsection{Shift operators and distributional kernels}\label{FlatShiftOperator}

We recall some results of \cite{FOS}. We regard $\R^n$ as a subspace of
$\R^{n+1}$ using the embedding $\R^n \hookrightarrow \R^{n+1}$ defined by $x
\mapsto (0,x)$. Elements of $\R^{n+1}$ are written in the form $(r,x)$. We
regard $\R^{n+1}$ as a Riemannian manifold with the flat Euclidian metric
$g_0$. The metric $g_0$ is the conformal compactification of the upper
half-space $(\R^{n+1}_+, g_{\hyp})$ with the hyperbolic metric $g_{\hyp} \st
r^{-2} g_0$. The hyperbolic metric $g_{\hyp}$ is a Poincar\'e metric $g_+$ in
normal form relative to the flat metric $h_0$ on the boundary $\R^n$ of
$\R^{n+1}_+$.

Let
\begin{equation}\label{eq: DistrKernels}
   K^+_{\lambda,\nu}(r,x) \st \abs{r}^{\lambda+\nu-n-1} (\abs{x}^2+r^2)^{-\nu}
   \quad \mbox{and} \quad
   K^-_{\lambda,\nu}(r,x) \st r K^+_{\lambda-1,\nu}(r,x)
\end{equation}
be the distributional Schwartz kernels studied in \cite{KS1,MO}. The maps
$$
   \varphi \mapsto \int_{\R^{n+1}} K^+_{\lambda,\nu}(r,x-y) \varphi(r,x) dr dx
$$
define operators $C_c^\infty(\R^{n+1}) \to C^\infty(\R^n)$ which are
equivariant with respect to principal series representations of the conformal
group of the subspace $\R^n \hookrightarrow \R^{n+1}$. Sometimes they are
referred to as {\em symmetry breaking operators} \cite{KS1}. Moreover, we set
\begin{equation}\label{eq:FlatShiftOperator}
    P(\lambda) \st r \Delta - (2\lambda-n-3) \partial_r: C^\infty(\R^{n+1}) \to C^\infty(\R^{n+1})
\end{equation}
(see \cite[(3.5)]{FOS}), where $\Delta$ is the Laplacian of the flat metric
$g_0$ on $\R^{n+1}$. The operator $P(\lambda)$ is equivariant with respect to
principal series representations for the conformal group of $\R^n$ (see
\eqref{flat-shift-equivariance}). The following result motivates us to refer to
this operator as a {\em shift operator} for the kernel
$K^\pm_{\lambda,\nu}(r,x)$ \cite[Theorem $3.5$]{FOS}.

\begin{prop}\label{FlatShiftVersion} The operator $P(\lambda)$ shifts the $\lambda$-parameter
of the kernels $K^\pm_{\lambda,\nu}(r,x)$, i.e.,
\begin{equation*}
   P(\lambda)K^\pm_{\lambda,\nu}(r,x)= (\lambda\!+\!\nu\!-\!n\!-\!1)(\nu\!-\!\lambda\!+\!1)
   K^\mp_{\lambda-1,\nu}(r,x).
\end{equation*}
\end{prop}

It was proven in \cite[Theorem 4.2]{FOS} that composition of shifted versions
of $P(\lambda)$ recover the conformal symmetry breaking differential operators
$D_N(\lambda)$ studied in \cite{J1,koss,KS1}. For a related representation
theoretical proof of this fact we refer to \cite{C}.\footnote{The operator
$P(\lambda)$ appears as (4.6) in \cite{C}.}

\subsection{The operator $\D_\lambda(g_+)$}\label{D}

Assume we are in the setting of Section \ref{setting}. In particular, $(M^n,h)$
is a Riemannian manifold and $g_+$ is a Poincar\'e metric on $M_+^\circ$ in
normal form relative to $h$. Following Graham and Zworski \cite[(4.4)]{GZ}, we
introduce a differential operator $\D_\lambda(g_+)$ on $C^\infty(M_+^\circ)$
which plays a basic role in the construction of the solution operators
$\T_{2j}(h;\lambda)$ (see Section \ref{PEandGJMS}). This operator is defined by
the equation
\begin{equation}
   r^{-(n-\lambda+1)}[\Delta_{g_+}+\lambda(n-\lambda)] r^{n-\lambda} = -\D_\lambda(g_+), \;
   \lambda \in \C
\end{equation}
and has the explicit form
\begin{equation}\label{eq:GZOperator}
   \D_\lambda(g_+) = -r \partial_r^2 + \left(2\lambda\!-\!n\!-\!1-\!\frac{r}{2}
   \tr(h_r^{-1} \dot{h}_r)\right) \partial_r
   - \frac{n\!-\!\lambda}{2} \tr(h_r^{-1}\dot{h}_r) - r \Delta_{h_r}.
\end{equation}
In the case of the hyperbolic upper-half space $(\R^{n+1}_+,g_{\hyp})$, it is
given by
\begin{equation*}
   \D_\lambda(g_{\hyp}) = - r\partial_r^2 + (2\lambda\!-\!n\!-\!1) \partial_r - r \Delta_{h_0}
   = - r \Delta_{g_0} + (2\lambda\!-\!n\!-\!1) \partial_r.
\end{equation*}
Comparing this formula with \eqref{eq:FlatShiftOperator}, yields the relation
\begin{equation}\label{PD}
   P(\lambda) = -\D_{\lambda-1}(g_{\hyp}).
\end{equation}

\subsection{The degenerate Laplacian}\label{Laplace-deg}

We recall the definition of the degenerate Laplacian introduced by Gover and
Waldron in \cite{GW}. Let $(X,c)$ be a $n+1$-dimensional conformal Riemannian
manifold equipped with a scale $\sigma \in C^\infty(X)$. Associated to these
data, we define the operator
\begin{equation}\label{eq:DegenerateLaplace}
   u \mapsto -\sigma \Delta_g u + (n\!+\!2\omega\!-\!1) \left[g(d\sigma,du)
   -\frac{\omega}{n+1}\Delta_g(\sigma) u\right]
   -\frac{2\omega}{n+1}(n\!+\!\omega)\sigma \J(g) u
\end{equation}
on $C^\infty(X)$ (see \cite[(2.9)]{GW}). Here $g$ is a metric in the conformal class
$c$ and $\omega\in\C$. The above operator will be denoted by $\ID
[g,\sigma;\omega]$. The notation $\IDD$ for the degenerate Laplacian reflects its
definition as a scalar product (in a tractor bundle) of a scale tractor $I$ and the
tractor operator $D$ mapping functions to tractors. We shall not go here into the
definitions of the relevant concepts of tractor calculus. Let $M = \sigma^{-1}(0)$
be the zero-locus of $\sigma$. We assume that $\sigma$ is a defining function for
the hypersurface $M$. Let $\iota: M \hookrightarrow X$ denote the embedding of $M$.

It follows from the definition that the operator $\iota^* \ID [g,\sigma;\omega]$
degenerates to the first-order operator
$$
   u \mapsto (n\!+\!2\omega\!-\!1) \iota^* \left[g(d\sigma,du) - \frac{\omega}{n\!+\!1}\Delta_g(\sigma) u \right].
$$
If $\sigma^{-2} g$ has constant scalar curvature $-n(n-1)$, it follows that
$|\grad_g(\sigma)|^2=1$ on $M$ and this operator reduces to the conformally
covariant Robin type boundary operator
$$
   u \mapsto (n\!+\!2\omega\!-\!1) \iota^* (\nabla_{\grad_g(\sigma)} - \omega H_g) u,
$$
where $H_g$ is the mean curvature of $M$ (\cite[Section 3.1]{Gover-AE},
\cite[Section 6.2]{J1}).

By its very definition in terms of tractor calculus, the operator $\IDD$ satisfies a
conformal covariance property. For the convenience of the reader, we provide an
independent proof of that basic property.

\begin{prop}\label{CTDL} The degenerate Laplacian satisfies
\begin{equation}\label{CT-ID}
   \ID[e^{2\varphi}g,e^{\varphi}\sigma;\omega] \circ e^{\omega\varphi}
   = e^{(\omega-1)\varphi} \circ \ID[g,\sigma;\omega], \; \omega \in \C
\end{equation}
for all metrics $g$, scales $\sigma \in C^\infty(X)$ and $\varphi \in
C^\infty(X)$. Moreover,
\begin{equation}\label{natural}
   \ID[\kappa^*(g),\kappa^*(\sigma);\omega] = \kappa^* \circ \ID[g;\sigma;\omega] \circ \kappa_*
\end{equation}
for any diffeomorphism $\kappa$ of $X$. Here $\kappa_*\st (\kappa^{-1})^*$.
\end{prop}

\begin{proof} The first claim follows by a straightforward computation involving the
standard identities
\begin{align}
   \Delta_{e^{2\varphi}g} & =e^{-2\varphi}[\Delta_g+(n-1)g(\grad_g(\varphi),\grad_g(\cdot))], \label{A} \\
   \J(e^{2\varphi} g) & = e^{-2\varphi} \left[\J(g) - \Delta_g\varphi-\frac{n\!-\!1}{2}\abs{d\varphi}_g^2\right],
   \notag \\
   \Delta_g(e^{\omega \varphi}) & = \omega e^{\omega\varphi}\Delta_g \varphi+\omega^2 e^{\omega \varphi}
   \abs{d\varphi}_g^2, \notag \\
   \grad_{e^{2\varphi}g}(f) & = e^{-2\varphi}\grad_g(f), \notag \\
   \Delta_g(f_1 f_2) & = \Delta_g(f_1) f_2 + f_1 \Delta_g(f_2) + 2g(\grad_g(f_1),\grad_g(f_2)), \notag \\
   \grad_g(f_1 f_2) & = \grad_g(f_1) f_2 + f_1 \grad_g(f_2) \notag
\end{align}
for any $f,f_1,f_2\in C^\infty(X)$. These relations imply
\begin{align*}
   & e^{-(\omega-1)\varphi} \circ \ID[e^{2\varphi}g,e^{\varphi}\sigma;\omega] \circ e^{\omega\varphi}\\
   & = \ID[g,\sigma;\omega]+ \left[-\omega - \frac{\omega}{n\!+\!1} (2\omega\!+\!n\!-\!1) +
   \frac{2\omega}{n\!+\!1}(n\!+\!\omega)\right] \sigma \Delta_g\varphi\\
   & + \big[-(2\omega\!+\!n\!-\!1)+(2\omega+n-1)\big]\sigma g(\grad_g(\varphi),\grad_g(\cdot))\\
   & + \left[-\omega(\omega\!+\!n\!-\!1)+\omega(2\omega\!+\!n\!-\!1)-\frac{n\omega (2\omega\!+\!n\!-\!1)}{n\!+\!1}
   + \frac{\omega (n\!-\!1)(n\!+\!\omega)}{n\!+\!1}\right]\sigma \abs{d\varphi}_g^2\\
   & + \left[\omega(2\omega\!+\!n\!-\!1)-(n\!+\!1)\frac{\omega(2\omega\!+\!n\!-\!1)}{n\!+\!1}\right]
   g(\grad_g(\varphi),\grad_g(\sigma))\\
   & =\ID[g,\sigma;\omega].
\end{align*}
The second claim is immediate by construction. The proof is complete.
\end{proof}

For $(X,g)$ being the flat Euclidean space $(\R^{n+1},g_0)$ and the scale $r$
with zero-locus $\R^n$ (as in Section \ref{FlatShiftOperator}), the degenerate
Laplacian $\IDD$ reduces to
\begin{equation*}
   \ID[g_0,r;\lambda\!-\!n\!-\!1] = -r\Delta_{g_0} + (2\lambda\!-\!n\!-\!3) \partial_r.
\end{equation*}
Hence
\begin{equation}\label{PID}
   P(\lambda) = - \ID[g_0;r,\lambda\!-\!n\!-\!1],
\end{equation}
i.e., the shift operator $P(\lambda)$ is a special case of $\IDD$.

\section{A curved version of the shift operator}\label{BSOperator}

Now we extend the definition of the shift operators $P(\lambda)$ (in Section
\ref{FlatShiftOperator}) to the setting of Section \ref{PEandGJMS}. Thus, we
assume that $(M^n,h)$ is a Riemannian manifold of dimension $n\ge 3$ and we let
$g_+ = r^{-2}(dr^2+h_r)$ be an even Poincar\'e metric in normal form relative
to $h$ on $M_+^\circ = (0,\varepsilon) \times M$ (for some $\varepsilon > 0$).
In particular, that means that, for odd $n$, $h_r$ has an expansion
$$
   h_r = h +r^2 h_2 + r^4 h_4 + \cdots,
$$
where all coefficients $h_{2j}$ are determined by $h$, and for even $n$ in the
expansion
$$
   h_r = h +r^2 h_2 + \cdots + r^n h_n + r^{n+2} h_{n+2} + \dots
$$
the coefficients $h_{\le n-2}$ and $\tr_h(h_n)$ are determined by $h$. We
recall the notation $\bar{g} = r^2 g_+$. The following definition corresponds
to \eqref{shift-curved}.

\begin{defn}\label{sop} The second-order differential
operator
\begin{equation}\label{eq:BSOperators}
   S(g_+;\lambda) \st r\Delta_{\bar{g}}-(2\lambda\!-\!n\!+\!1)\partial_r -
   \frac{1}{2}(\lambda\!-\!n\!+\!1) \tr(h_r^{-1}\dot{h}_r)
\end{equation}
is called a {\em shift operator}. Here the dot denotes the derivative with
respect to $r$.
\end{defn}

The notion {\em shift operator} is motivated by Theorem \ref{CurvedBS}.

We shall regard $S(g_+;\lambda)$ as an operator $C^\infty(M_+^\circ) \to
C^\infty(M_+^\circ)$. Since $\bar{g}$ is a smooth metric on $M_+ = [0,
\varepsilon) \times M$, the shift operators may also be regarded as operators
on $C^\infty(M_+)$, and mapping properties on $M_+^\circ$ naturally extend to
$M_+$.

We emphasize that the shift operator $S(g_+;\lambda)$ is not completely
determined by the Taylor coefficients of $h_r$. The composition $\iota^*
S(g_+;\lambda)$ degenerates to a first-order operator.

By \eqref{v-trace}, the shift operator can also be written in the form
\begin{equation}\label{shift-op-2}
   S(g_+;\lambda) \st r \Delta_{\bar{g}}-(2\lambda\!-\!n\!+\!1)\partial_r -
   (\lambda\!-\!n\!+\!1) \dot{v}(r)/v(r).
\end{equation}
The following result describes the relations among the three operators
$\D_\lambda$, $\ID[\cdot;\lambda]$ and $S(\cdot;\lambda)$.

\begin{prop}\label{EquivalentOperators} It holds
\begin{equation}\label{Sand D}
    S(g_+;\lambda) = -\D_{\lambda+1}(g_+).
\end{equation}
Moreover, we have
\begin{equation}\label{SandID}
    S(g_+;\lambda) = - \ID[\bar{g};r,\lambda-n+1]
\end{equation}
if $g_+$ is Einstein.
\end{prop}

\begin{proof} We recall the expressions \eqref{eq:PELaplacian} and
\eqref{eq:GZOperator} for $\Delta_{g_+}$ and $\D_\lambda(g_+)$. Moreover, by
\eqref{eq:CCLaplacian}, we have
\begin{equation*}
    \Delta_{\bar{g}} = r^{-2}(\Delta_{g_+} + (n-1) r\partial_r).
\end{equation*}
Combining these formulas, we obtain
\begin{align*}
   S(g_+;\lambda) & = r^{-1}(\Delta_{g_+} + (n-1)r\partial_r) - (2\lambda\!-\!n\!+\!1)\partial_r -
   \frac 12(\lambda\!-\!n\!+\!1) \tr(h^{-1}_r\dot{h}_r) \\
   & = r\Delta_{h_r} + r\partial_r^2 + \frac12 \tr(h^{-1}_r\dot{h}_r) r\partial_r -
   (2\lambda\!-\!n\!+\!1)\partial_r - \frac 12(\lambda\!-\!n\!+\!1) \tr(h^{-1}_r\dot{h}_r)\\
   & = r\Delta_{h_r} + r\partial_r^2 - \left(2\lambda\!-\!n\!+\!1-\!\frac{r}{2} \tr(h^{-1}_r\dot{h}_r)\right)
   \partial_r - \frac 12(\lambda\!-\!n\!+\!1) \tr(h^{-1}_r\dot{h}_r) \\
   & = -\D_{\lambda+1}(g_+).
\end{align*}
This proves the first claim. The second claim follows by evaluating $\IDD$ for
$g=\bar{g}$ and $\sigma=r$. In particular, using
$$
   \Delta_{\bar{g}}(r) = \frac{1}{2} \tr(h^{-1}_r\dot{h}_r)
$$
and the relation
\begin{equation}\label{Jbar}
   \J(\bar{g}) = -\frac{1}{2r}\tr(h^{-1}_r \dot{h}_r)
\end{equation}
(\cite[(6.11.8)]{J1}), we find
\begin{align*}
   \ID[\bar{g};r,\omega] & =-r\Delta_{\bar{g}}
   + (n\!+\!2\omega\!-\!1) \left(\partial_r-\frac{\omega}{n+1}\Delta_{\bar{g}}(r)\right)
   -\frac{2\omega}{n+1}(n\!+\!\omega)r \J(\bar{g})\\
   & =-r\Delta_{\bar{g}}+(n\!+\!2\omega\!-\!1)\partial_r+\frac{\omega}{2}\tr(h^{-1}_r\dot{h}_r).
\end{align*}
Hence $S(g_+;\lambda)= -\ID[\bar{g};r,\lambda-n+1]$. The proof is complete.
\end{proof}

\begin{remark} The identity \eqref{Jbar} is not valid for general Poincar\'e metrics $g_+$ since these
are only asymptotically Einstein. For the convenience of the reader, we insert
a proof of \eqref{Jbar} for Einstein $g_+$ which also clarifies the
modification in the general case. We recall the well-known transformation
formula
\begin{equation}\label{reverse}
   \tau(\hat{g}) = e^{-2\varphi} \tau(g) - 2n \Delta_{\hat{g}}(\varphi) + n(n\!-\!1)
   |d\varphi|_{\hat{g}}^2, \quad \hat{g} = e^{2\varphi}g
\end{equation}
for the scalar curvature. We apply \eqref{reverse} to $\hat{g} = dr^2 \!+\!
h_r$, $g = r^{-2}(dr^2\!+\!h_r)$ and $\varphi = \log r$. For Einstein $g$,
i.e., $\Ric(g) + ng = 0$, we have $\tau(g) = -n(n\!+\!1)$. Hence we find
\begin{align*}
   \tau(dr^2 \!+\! h_r) & = - \frac{n(n\!+\!1)}{r^2} - 2n \Delta_{dr^2+h_r}(\log r)
   + n(n\!-\!1) |d\log r|^2_{dr^2+h_r} \nonumber \\
   & = - \frac{n(n\!+\!1)}{r^2} - 2n \left(\partial_r^2 (\log r)
   + \frac{1}{2} \tr (h_r^{-1} \dot{h}_r) \partial_r (\log r)\right)
   + \frac{n(n\!-\!1)}{r^2}  \nonumber \\
   & = - n \tr (h_r^{-1} \dot{h}_r) \frac{1}{r}.
\end{align*}
Hence
\begin{equation*}\label{tau-taylor}
   \J(dr^2\!+\!h_r) = -\frac{1}{2r} \tr (h_r^{-1} \dot{h}_r).
\end{equation*}
We also note that, in \cite{GW}, the authors deal with Einstein metrics outside
hypersurfaces in Riemannian manifolds. In particular, the calculation at the
end of \cite[Section 5]{GW} employs the formula \eqref{Jbar}.
\end{remark}

\begin{remark}\label{conjugation} By the first part of Proposition
\ref{EquivalentOperators}, we have
\begin{equation}\label{NewForm}
   S(g_+;\lambda) = r^{\lambda-n} \circ (\Delta_{g_+}+(\lambda\!+\!1)(n\!-\!\lambda\!-\!1))
   \circ r^{n-\lambda-1}
\end{equation}
as an identity of operators acting on $C^\infty(M_+^\circ)$. Sometimes, this
formula for $S(g_+;\lambda)$ will be more convenient to work with than the
original definition \eqref{eq:BSOperators}. However, the original definition
has the advantage that it clearly shows that $S(g_+;\lambda)$ acts on smooth
functions on $M_+$. In order to illustrate the convenience of the conjugation
formula \eqref{NewForm}, we note that it yields a slightly more conceptual
proof of \eqref{Jbar}. We use the fact that for Einstein $g_+$
$$
   P_2(g_+) = \Delta_{g_+} + m(m-1)
$$
with $m = \frac{n+1}{2}$. Hence
\begin{equation*}
   P_2(\bar{g}) = r^{-m-1} P_2(g_+) r^{m-1} = r^{-m-1} (\Delta_{g_+} + m(m-1)) r^{m-1}
   = r^{-1} S\left(g_+;\frac{n-1}{2}\right).
\end{equation*}
This identity is the special case $N=1$ of Theorem \ref{TC}. On the function
$1$ it yields \eqref{Jbar}.
\end{remark}

\begin{remark}\label{AE} For even $n$, the condition $\Ric(g_+)+ng_+ = O(r^{n-2})$
implies $\tau(g_+) = -n(n+1) + O(r^n)$. Hence an extension of the arguments in
Remark \ref{conjugation} gives
\begin{align*}
   P_2(g_+) = \Delta_{g_+} + m(m-1) + O(r^n) \quad \mbox{and} \quad
   P_2(\bar{g}) = r^{-1} S\left(g_+;\frac{n-1}{2}\right) + O(r^{n-2}).
\end{align*}
The trace condition \eqref{VT} improves the estimate of the scalar curvature to
$\tau(g_+) = -n(n+1) + o(r^n)$ and we obtain
$$
   r P_2(\bar{g}) = S \left(g_+;\frac{n-1}{2}\right) + o(r^{n-1}).
$$
This is the special case $N=1$ of Theorem \ref{TC} for Poincar\'e metrics.
\end{remark}

We continue with the discussion of the {\em shift property} of the shift
operators. Its description requires some more notation. We recall that
$M_+^\circ = (0,\varepsilon) \times M$ and $u \in C^\infty(M_+^\circ)$ be a
solution to the equation
\begin{equation}\label{eigen}
   \Delta_{g_+} u + \nu(n-\nu) u = 0 \quad \mbox{for $\nu \in n/2 + i\R$, $\nu \ne n/2$}
\end{equation}
with {\em boundary value} $f \in C^\infty(M)$ (see Section \ref{PEandGJMS}).
Such an {\em exact} eigenfunction can be constructed as described in Section
\ref{PEandGJMS} by regarding $M_+$ as part of a conformally compact manifold.
Then $u$ gives rise to the family (see \cite[Section 6.6]{J1})
\begin{equation}\label{eq:TheDistribution}
   \lambda \mapsto M_u(r;\lambda) \in C^\infty(M_+^\circ),
   \quad M_u(r;\lambda) \st r^{\lambda-n+1}u, \; \lambda \in \C.
\end{equation}
We shall consider $M_u(r;\lambda)$ as a family of functions on $M_+^\circ$ as
well as a family of distributions on $M_+$. In the latter case, we have
$$
   \langle M_u(r;\lambda), \varphi \rangle = \int_{M_+} r^{\lambda-n+1} u
   \varphi dvol(\bar{g}), \quad \Re(\lambda) > n/2-2,
$$
where the test functions $\varphi$ are in $C_c^\infty(M_+)$ and the condition
$\Re(\lambda) > \frac{n}{2}-2$ guarantees the convergence of the integral. Then
$M_u(r;\lambda)$ will be understood as a meromorphic family of distributions.
For simplicity, we shall denote that family of distributions also by
$M_u(r;\lambda)$. Finally, we introduce the multiplication operators
\begin{equation}\label{mult}
   M_r \st r\cdot.
\end{equation}

\begin{remark}\label{Muaskernel} We consider the distribution $M_u(r;\lambda)$
in the flat case. Let $u \in C^\infty(\HS^{n+1})$ be an eigenfunction of the
Laplacian on the hyperbolic upper-half space $\HS^{n+1}$. We assume that $u$
can be written as Helgason's Poisson transform \cite{Hel}
\begin{equation}\label{PT}
   u(r,x) = \int_{\R^n} \left(\frac{r}{|x-y|^2+r^2}\right)^\nu f(y) dy
\end{equation}
of $f \in C_0^\infty(\R^n)$, say. Then
\begin{align*}
   \langle M_u(r;\lambda), \varphi \rangle & = \int_{\HS^{n+1}}
   r^{\lambda-n+1} u(r,x) \varphi(r,x) dr dx \\
   & = \int_{\HS^{n+1}} \int_{\R^n} r^{\lambda-n+1}
   \left(\frac{r}{|x-y|^2+r^2}\right)^\nu f(y) \varphi(r,x) dy dr dx  \\
   & = \int_{\R^n} f(y) \left( \int_{\HS^{n+1}} K^+_{\lambda-2,\nu}(r,x-y) \varphi(r,x) dx
   dr \right) dy.
\end{align*}
The latter formula shows in which sense \eqref{eq:TheDistribution} can be
regarded as a generalization of the distributional kernels
$K_{\lambda,\nu}^\pm(r,x)$ (see \eqref{eq: DistrKernels}). We also note that
the asymptotic expansion of $u$ given by \eqref{PT} is of the form
$$
   r^\nu \sum_{j \ge 0} r^{2j} a_{2j}(x) +
   r^{n-\nu} c(\lambda) \sum_{j\ge 0} r^{2j} b_{2j}(x), \; r \to 0,
$$
where $b_0=f$ and $c(\lambda)$ is Harish-Chandra's $c$-function. In particular,
the boundary value $f$ appears in the coefficient of $r^{n-\lambda}$. Note that
this notion of boundary value slightly differs from that used in Section
\ref{PEandGJMS} by the coefficient $c(\lambda)$.
\end{remark}

Now we are ready to state and prove the {\em shift property} of the operators
$S(g_+;\lambda)$. The following result generalizes \cite[Theorem $3.5$]{FOS}.
Following the terminology of \cite{FOS}, it may be referred to as a
Bernstein-Sato identity.

\begin{theorem}\label{CurvedBS} Let $(M^n,h)$ be a Riemannian manifold. Let
$u$ be a solution of \eqref{eigen}. Then the multiplication operator $M_r$ and
the shift operator $S(g_+;\lambda)$ shift the $\lambda$-parameter when acting
on the function $M_u(r;\lambda)$, i.e.,
\begin{align*}
   M_r (M_u(r;\lambda)) & = M_u(r;\lambda+1),\\
   S(g_+;\lambda)(M_u(r;\lambda)) & =(\lambda\!+\!\nu\!-\!n\!+\!1)(\nu\!-\!\lambda\!-\!1)M_u(r;\lambda-1).
\end{align*}
\end{theorem}

\begin{proof} The first claim is obvious from the definitions. Using
\eqref{NewForm} and the fact that $u$ satisfies \eqref{eigen}, we compute
\begin{align*}
   S(g_+;\lambda)(M_u(r;\lambda)) & = r^{\lambda-n} \Delta_{g_+}(r^{n-\lambda-1} r^{\lambda-n+1} u)
   + (\lambda\!+\!1)(n\!-\!\lambda\!-\!1) r^{\lambda-n} u\\
   & = -\nu(n\!-\!\nu) r^{\lambda-n}u + (\lambda\!-\!n\!+\!1)(-\lambda\!-\!1) r^{\lambda-n}u\\
   & = (\lambda\!+\!\nu\!-\!n\!+\!1)(\nu\!-\!\lambda\!-\!1) M_u(r;\lambda-1).
\end{align*}
The proof is complete.
\end{proof}

From Theorem \ref{CurvedBS} we easily deduce the poles of the meromorphic
continuation of the holomorphic family of distributions
$$
   C_c^\infty(M_+) \ni \varphi \mapsto \int_{M_+} M_u(r;\lambda) \varphi dvol(\bar{g}), \quad
   \Re(\lambda) > \frac{n}{2}-2.
$$

\begin{cor}\label{mero} The family $\lambda \mapsto M_u(r;\lambda)$ is meromorphic
with generically simple poles at\footnote{Here $\nu$ is generic if both ladders
of poles in \eqref{eq:Poles} do not intersect. For $\nu \in \frac{n}{2}+i\R$
being {\em generic} it suffices to assume that $\nu \ne \frac{n}{2}$.}
\begin{equation}\label{eq:Poles}
   \lambda = -\nu\!+\!n\!-\!2\!-\!N \quad \mbox{and} \quad \lambda=\nu\!-\!2\!-\!N
\end{equation}
for $N\in\N_0$.
\end{cor}

\begin{proof} Indeed, Theorem \ref{CurvedBS} implies
$$
   M_u(r;\lambda) = \frac{1}{(\lambda\!+\!\nu\!-\!n\!+\!2)(\nu\!-\!\lambda\!-\!2)}
   S(g_+;\lambda+1)(M_u(r;\lambda+1)).
$$
The distribution on the left-hand side of this identity is holomorphic in the
half-plane $\Re(\lambda) > \frac{n}{2}-2$. The right-hand side provides a
meromorphic continuation to $\Re(\lambda) > \frac{n}{2}-3$ with simple poles in
$\lambda=-\nu+n+2$ and $\lambda=\nu-2$. The assertion follows by a repeated
application of that argument.
\end{proof}

Alternatively, Corollary \ref{mero} can be proved by directly inserting the
asymptotic expansion of $u$ into the integral which defines the distribution
$M_u(r;\lambda)$ for $\Re(\lambda) > \frac{n}{2}-2$. The above argument using
Theorem \ref{CurvedBS} is more conceptual, however.

Next, we discuss a conformal transformation law for shift operators. Let
$\hat{h} = e^{2\varphi} h$ be conformally equivalent to $h$. We choose an even
Poincar\'e metric $g_+$ in normal form relative to $h$ and let $\hat{g}_+ =
\kappa^*(g_+)$ be a related even Poincar\'e metric in normal form relative to
$\hat{h}$; for the construction of the diffeomorphism $\kappa$ we refer to
Section \ref{PEandGJMS}. In these terms, we have the following result.

\begin{prop}\label{ConformalCovarianceP}
\begin{equation*}
   S(\hat{g}_+;\lambda) = \left(\frac{r}{\kappa^*(r)}\right)^{\lambda-n} \circ \kappa^* \circ
   S(g_+;\lambda) \circ \kappa_* \circ
   \left(\frac{r}{\kappa^*(r)}\right)^{n-\lambda-1}.
\end{equation*}
\end{prop}

\begin{proof} By \eqref{NewForm}, we have
\begin{equation*}
   S(\hat{g}_+;\lambda) = r^{\lambda-n} \Delta_{\hat{g}_+} r^{n-\lambda-1}
   + (\lambda\!+\!1)(n\!-\!\lambda\!-\!1) \frac{1}{r}.
\end{equation*}
But $\kappa^*(g_+)=\hat{g}_+$ implies $\Delta_{\hat{g}_+}=\kappa^*
\Delta_{g_+}\kappa_*$. Hence
\begin{align*}
   S(\hat{g}_+;\lambda) & = r^{\lambda-n} \kappa^* r^{-\lambda+n} r^{\lambda-n}
   \Delta_{g_+} r^{n-\lambda-1} r^{\lambda-n+1} \kappa_*(r^{n-\lambda-1}\cdot)
   + (\lambda\!+\!1)(n\!-\!\lambda\!-\!1)\frac{1}{r}\\
   & =r^{\lambda-n} \kappa^* (r^{-\lambda+n}) \kappa^* r^{\lambda-n} \Delta_{g_+}
   r^{n-\lambda-1}
   \kappa_* \kappa^*(r^{\lambda-n+1}) r^{n-\lambda-1} +(\lambda\!+\!1)(n\!-\!\lambda\!-\!1)\frac{1}{r}\\
   & = \left(\frac{r}{\kappa^*(r)}\right)^{\lambda-n} \kappa^* ( r^{\lambda-n} \Delta_{g_+}
   r^{n-\lambda-1}) \kappa_* \left(\frac{r}{\kappa^*(r)}\right)^{n-\lambda-1}
   + (\lambda\!+\!1)(n\!-\!\lambda\!-\!1)\frac{1}{r}.
\end{align*}
Now the obvious identity
\begin{equation*}
   \frac{1}{r} = \left(\frac{r}{\kappa^*(r)}\right)^{\lambda-n} \circ \kappa^* \circ
   \frac{1}{r} \circ \kappa_* \circ \left(\frac{r}{\kappa^*(r)}\right)^{n-\lambda-1}
   \end{equation*}
implies
\begin{equation*}
   S(\hat{g}_+,r;\lambda) = \left(\frac{r}{\kappa^*(r)}\right)^{\lambda-n} \kappa^* S(g_+;\lambda) \kappa_*
   \left(\frac{r}{\kappa^*(r)}\right)^{n-\lambda-1}.
\end{equation*}
This completes the proof.
\end{proof}

Theorem \ref{TA} is a direct consequence of this result.

\begin{remark}\label{ConformalTrafoP} For Einstein $g_+$, Proposition
\ref{ConformalCovarianceP} also follows by combining the conformal covariance
of $\IDD$ (Proposition \ref{CTDL}) with the identification of $S(g_+;\lambda)$
as a degenerate Laplacian (Proposition \ref{EquivalentOperators}). Indeed, for
$\hat{h} = e^{2\varphi} h$, we write $g_+ = r^{-2} (dr^2 + h_r)$ and $\hat{g}_+
= r^{-2} (dr^2 + \hat{h}_r)$ with $\hat{g}_+ = \kappa^*(g_+)$ and $h_0 =h$,
$\hat{h}_0 = \hat{h}$. Then
\begin{equation}\label{trans}
   \kappa^*(dr^2 + h_r) = \left(\frac{\kappa^*(r)}{r}\right)^2(dr^2+\hat{h}_r).
\end{equation}
Now assume that $g_+$ is Einstein. Then
\begin{align*}
   S(\hat{g}_+;\lambda) & = -\ID[dr^2 \!+\! \hat{h}_r,r;\lambda\!-\!n\!+\!1] & (\mbox{by \eqref{SandID}}) \\
   & = -\ID\left[\left(\frac{\kappa^*(r)}{r}\right)^{-2} \kappa^*(dr^2 \!+\! h_r),
   \left(\frac{\kappa^*(r)}{r}\right)^{-1} \kappa^*(r);\lambda\!-\!n\!+\!1\right] & (\mbox{by \eqref{trans}}) \\
   & = -\left(\frac{r}{\kappa^*(r)}\right)^{\lambda-n} \ID\left[\kappa^*(dr^2 \!+\! h_r),
   \kappa^*(r);\lambda\!-\!n\!+\!1\right]
   \left(\frac{r}{\kappa^*(r)}\right)^{\lambda-n+1} & (\mbox{by \eqref{CT-ID}}) \\
   & = -\left(\frac{r}{\kappa^*(r)}\right)^{\lambda-n} \kappa^* \ID\left[dr^2 \!+\! h_r,r;
   \lambda\!-\!n\!+\!1\right] \kappa_*
   \left(\frac{r}{\kappa^*(r)}\right)^{\lambda-n+1} & (\mbox{by \eqref{natural}}) \\
   & = \left(\frac{r}{\kappa^*(r)}\right)^{\lambda-n} \kappa^* S(g_+;\lambda) \kappa_*
   \left(\frac{r}{\kappa^*(r)}\right)^{\lambda-n+1}. & (\mbox{by \eqref{SandID}})
\end{align*}
A general Poincar\'e metric $g_+$ is only approximate Einstein. For such a
metric, the identification of $S(g_+;\lambda)$ with a degenerate Laplacian
holds true only up to an error term. This leads to error terms in the above
calculation and in the resulting conformal transformation law.
\end{remark}

In Section \ref{OnJuhlsFormulae}, we shall prove a relation between
compositions of shift operators and residue families. For that purpose, we need
a relation among the operator $S(g_+;\lambda)$ and its formal adjoint
$S^*(g_+;\lambda)$ with respect to the scalar product defined by $\bar{g}$.

\begin{prop}\label{AdjointOfP} The formal adjoint of the operator $S(g_+;\lambda)$ acting
on $C^\infty_c(M_+^\circ)$ with respect to the scalar product defined by
$\bar{g}$ is given by
\begin{equation*}
    S^*(g_+;\lambda) = S(g_+;n\!-\!\lambda\!-\!2).
\end{equation*}
\end{prop}

\begin{proof} Let $f_1,f_2 \in C_c^\infty(M_+^\circ)$. Then, using $dvol(\bar{g})=r^{n+1}dvol(g_+)$
and \eqref{NewForm}, we have
\begin{align*}
   \int_{M_+} (S(g_+;\lambda)f_1) f_2 d vol(\bar{g})
   &=\int_{M_+} r^{\lambda-n}[\Delta_{g_+}\!+\!(\lambda\!+\!1)(n\!-\!\lambda\!-\!1)](r^{n-\lambda-1}f_1) f_2
   r^{n+1} dvol(g_+)\\
   &=\int_{M_+} [\Delta_{g_+}\!+\!(\lambda\!+\!1)(n\!-\!\lambda\!-\!1)](r^{n-\lambda-1}f_1) r^{\lambda+1} f_2
   dvol(g_+)\\
   &=\int_{M_+} (r^{n-\lambda-1}f_1) [\Delta_{g_+}\!+\!(\lambda\!+\!1)(n\!-\!\lambda-1)](r^{\lambda+1} f_2)
   dvol(g_+)\\
   &=\int_{M_+} f_1 (r^{-\lambda-2}[\Delta_{g_+}\!+\!(\lambda\!+\!1)(n\!-\!\lambda\!-\!1)](r^{\lambda+1} f_2))
   dvol(\bar{g}).
\end{align*}
Thus
\begin{equation*}
   S^*(g_+;,\lambda) = r^{-\lambda-2} \circ (\Delta_{g_+}+(\lambda+1)(n-\lambda-1)) \circ r^{\lambda+1}.
\end{equation*}
Comparing this relation with \eqref{NewForm} completes the proof.
\end{proof}

In the proof of Theorem \ref{DeltaN}, we shall actually need the following
improved version of the latter result.

\begin{remark}\label{supplement} The relation
$$
   \int_{M_+} S(g_+;\lambda)(f_1) f_2 dvol(\bar{g}) = \int_{M_+} f_1 S(g_+;n\!-\!\lambda\!-\!2)(f_2) dvol(\bar{g})
$$
continues to be true for $f_2 \in C_c^\infty(M_+)$ and $f_1 \in
C^\infty(M_+^\circ)$ with an asymptotic expansion of the form
$$
   f_1(r,x) \sim \sum_{j \ge 0} r^{\nu+j} a_j(x), \; r \to 0, \; a_j \in C^\infty(M)
$$
with $\Re(\nu)>0$.
\end{remark}

\begin{proof} It suffices to prove the relations
$$
    \int_{M_+} \partial_r(f_1) f_2 dvol(\bar{g}) = - \int_{M_+} f_1 \partial_r(f_2) dvol(\bar{g})
    - \frac{1}{2} \int_{M_+} f_1 \tr(h_r^{-1}\dot{h}_r) f_2 dvol(\bar{g})
$$
and
$$
    \int_{M_+} (r \Delta_{\bar{g}})(f_1) f_2  dvol(\bar{g}) = \int_{M_+} f_1
    \left(r \Delta_{\bar{g}} + 2 \partial_r + \frac{1}{2} \tr (h_r^{-1} \dot{h}_r)\right)(f_2) dvol(\bar{g}).
$$
For the proof of the first relation we use \eqref{v-trace} and the observation
that the restriction of $f_1$ to the boundary vanishes. Green's formula and
\eqref{eq:CCLaplacian} imply the second relation. Here the boundary
contributions vanish since the restrictions of $f_1$ and $r
\partial_r(f_1)$ to the boundary both vanish.
\end{proof}

For later purpose, we need to understand the behavior of the composition of
$S(g_+;\lambda)$ with the multiplication by powers of $r$.

\begin{lem}\label{GeneralSL(2)} Let $f\in C^\infty(M_+^\circ)$ and $a\in\N$. Then
\begin{equation*}
   S(g_+;\lambda)(r^a f) = r^a S(g_+;\lambda\!-\!a)f - a(2\lambda\!-\!n\!+\!2\!-\!a) r^{a-1} f.
\end{equation*}
\end{lem}

\begin{proof} The proof is straightforward. The definition \eqref{eq:BSOperators}
and the formula \eqref{eq:CCLaplacian} for $\Delta_{\bar{g}}$ imply
\begin{align*}
   S(g_+;\lambda)(r^a f) & = r\Delta_{\bar{g}}(r^{a}f) - (2\lambda\!-\!n\!+\!1)\partial_r (r^a f)
   - \frac{1}{2}(\lambda\!-\!n\!+\!1) \tr(h_r^{-1}\dot{h}_r) r^a f\\
   & = r^a S(g_+;\lambda\!-\!a)f + [a(a\!-\!1)-(2\lambda\!-\!n\!+\!1)] r^{a-1}f\\
   & = r^a S(g_+;\lambda\!-\!a)f - a(2\lambda\!-n\!+\!2\!-\!a) r^{a-1}f.
\end{align*}
The proof is complete.
\end{proof}

Of course, this result is also true for $f \in C^\infty(M_+)$. In view of the
second part of Proposition \ref{EquivalentOperators}, it should be viewed as an
analog of the $sl(2)$-structure for the degenerate Laplacian $\IDD$ proved in
\cite[Lemma 3.1, Proposition 3.4]{GW}.

\section{A new formula for residue families}\label{OnJuhlsFormulae}

We recall that residue families $D_N^{res}(h;\lambda)$ are defined by
normalizations of the families $\delta_N(h;\lambda)$ (see
\eqref{eq:DefDeltaN}). In the present section, we express these families in
terms of shift operators $S(g_+;\lambda)$ (see \eqref{eq:BSOperators}) and
discuss some consequences.

For $N\in\N$, we define the family $S_N(g_+;\lambda)$ of shift operators on
$M_+^\circ$ by
\begin{equation}\label{eq:ThePpolynomial}
    S_N(g_+;\lambda) \st \underbrace{S(g_+;\lambda) \circ \cdots \circ S(g_+;\lambda+N-1)}_{N factors}.
\end{equation}
We also set $S_0(g_+;\lambda)=\id$.

Similarly, we define
\begin{equation}\label{eq:TheIDpolynomial}
   \ID_N[\bar{g},r;\lambda] \st
   \underbrace{\ID[\bar{g},r;\lambda] \circ \cdots \circ \ID[\bar{g},r;\lambda+N-1]}_{N factors}.
\end{equation}

Again, we shall regard $S_N(g_+;\lambda)$ as an operator $C^\infty(M_+^\circ) \to
C^\infty(M_+^\circ)$ and also as an operator on $C^\infty(M_+)$. By definition,
$S_N(g_+;\lambda)$ is a differential operator of order $2N$ with polynomial
coefficients in $\lambda$. As a polynomial in $\lambda$ it is of degree $N$.

We also recall that, for $N \in \N$ with $N \le n+1$ for even $n$, the families
$\delta_N(h;\lambda)$ are determined by $h$.

\begin{theorem}\label{DeltaN} Let $N\in\N$ with $N \le n+1$ for even $n$. Then
\begin{equation}\label{eq:DeltaN}
    \delta_N(h;\lambda) = \frac{1}{(-N)_N(2\lambda\!-\!n\!+\!1)_N} \iota^* S_N(g_+;\lambda)
\end{equation}
as an identity of meromorphic functions in $\lambda$ with values in operators
$C^\infty(M_+) \to C^\infty(M)$.
\end{theorem}

\begin{proof} Let $u$ be an eigenfunction with boundary value $f$ and
satisfying \eqref{eigen} with $\Re(\nu) = n/2$, $\nu \ne n/2$. We derive the
assertion from Theorem \ref{CurvedBS} and the identity
\begin{equation}\label{Delta0}
    \Res_{\lambda=-\nu-1} \left(\int_{M_+} M_u(r;\lambda\!+\!n\!-\!1)\varphi dvol(\bar{g})\right)
    = \int_M f \iota^*\varphi dvol(h)
\end{equation}
(see \eqref{eq:DefDeltaN} for $N=0$). In the following, it will be convenient
to use, for any $\lambda$-dependent operator $A(\lambda)$, the notation
\begin{align*}
    A((\lambda)_N) & \st A(\lambda)\circ A(\lambda+1)\circ\cdots\circ A(\lambda+N-1),\\
    A((\lambda)^N) & \st A(\lambda)\circ A(\lambda-1)\circ\cdots\circ A(\lambda-N+1).
\end{align*}
On the one hand, \eqref{eq:DefDeltaN} implies
\begin{equation}\label{eq:h1}
    \Res_{\lambda=-\nu-1-N} \left(\int_{M_+} M_u(r;\lambda\!+\!n\!-\!1)\varphi dvol(\bar{g})\right)
    = \int_M f \delta_N(h;\nu)\varphi dvol(h).
\end{equation}
On the other hand, using Theorem \ref{CurvedBS}, we obtain for $\Re(\lambda)
\notin -\frac{n}{2}-\N$
\begin{align*}
    M_u(r;\lambda\!+\!n\!-\!1)
    & = \frac{S(g_+;\lambda\!+\!n)(M_u(r;\lambda\!+\!n))}{(\lambda+\nu+1)(\nu-\lambda-n-1)} \\
    & = \frac{S(g_+;\lambda\!+\!n) S(g_+;\lambda\!+\!n\!+\!1)(M_u(r;\lambda\!+\!n\!+\!1))}
    {(\lambda+\nu+1)(\lambda+\nu+2)(\nu-\lambda-n-2)(\nu-\lambda-n-1)} \\
    & = \dots = \frac{S(g_+;(\lambda\!+\!n)_N)(M_u(r;\lambda\!+\!n\!+\!N\!-\!1))}{(\lambda+\nu+1)_N
    (\nu-\lambda-n-N)_N}.
\end{align*}
Now we take adjoints using Remark \ref{supplement}. This gives
\begin{multline*}
   \int_{M_+} M_u(r;\lambda\!+\!n\!-\!1)\varphi d vol(\bar{g})
   = \frac{1}{(\lambda+\nu+1)_N(\nu-\lambda-n-N)_N} \\
   \times \int_{M_+} M_u(r;\lambda\!+\!n\!+\!N\!-\!1)
   S^*(g_+;(\lambda\!+\!n\!+\!N\!-\!1)^N) \varphi dvol(\bar{g})
\end{multline*}
for $\Re(\lambda) > - \frac{n}{2}-1$. By the assumptions, the zeros of $\nu
\mapsto (\lambda+\nu+1)_N(\nu-\lambda-n-N)_N$ are simple for $\Re(\lambda)
> - \frac{n}{2}-1$. Hence \eqref{Delta0} implies
\begin{align*}
   & \Res_{\lambda=-\nu-1-N} \left(\int_{M_+} M_u(r;\lambda\!+\!n\!-\!1)\varphi
   dvol(\bar{g})\right) \\
   & = \frac{1}{(-N)_N(2\nu-n+1)_N} \int_M f \iota^*S^*(g_+;(-\nu\!+\!n\!-\!2)^N)\varphi dvol(h)\\
   & = \frac{1}{(-N)_N(2\nu-n+1)_N} \int_M f \iota^*S_N(g_+;\nu)\varphi dvol(h).
\end{align*}
Comparing this result with \eqref{eq:h1}, completes the proof for $\Re(\nu) =
n/2$, $\nu \ne n/2$. The assertion then follows by meromorphic continuation.
\end{proof}

Since the operator $S_N(g_+;\lambda)$ involves $2N$ derivatives in $r$, it is a
non-trivial observation that its composition with $\iota^*$ only depends on $h$
for appropriate choices of $N$. But this is an immediate consequence of Theorem
\ref{DeltaN} as long as $N \le n$ for even $n$. Therefore, it is justified to
introduce the notation
\begin{equation}\label{reduced-shift}
   \S_N(h;\lambda) \st \iota^* S_N(g_+;\lambda)
\end{equation}
for such $N$. However, for even $n$ and general $N \in \N$, the composition
$\iota^* S_N(g_+;\lambda)$ does not only depend on $h$ and the notation
$\S_N(h;\lambda)$ will not be used.

As a direct consequence of Theorem \ref{DeltaN}, we obtain the following
identification of residue families with operators $\S_N(h;\lambda)$.

\begin{cor}\label{RFvsSF} Assume that $N \in \N$ so that $2N \le n$ for even $n$. Then
\begin{equation}\label{ResFam-Shift-even}
   D_{2N}^{res}(h;\lambda) = \frac{1}{(-2N)_{N}(\lambda\!+\!\frac n2\!-\!2N\!+\!\frac{1}{2})_{N}}
   \S_{2N}(h;\lambda\!+\!n\!-\!2N)
\end{equation}
and
\begin{equation}\label{ResFam-Shift-odd}
   D_{2N+1}^{res}(h;\lambda) = \frac{1}{2(-2N\!-\!1)_{N+1}(\lambda\!+\!\frac n2\!-\!2N\!-\!\frac{1}{2})_{N+1}}
   \S_{2N+1}(h;\lambda\!+\!n\!-\!2N\!-\!1).
\end{equation}
\end{cor}

Some further comments on this result are in order.

We recall that all residue families $D_N^{res}(h;\lambda)$ are polynomials in
$\lambda$. Hence Corollary \ref{RFvsSF} shows that the zeros of the
denominators on the right-hand sides of \eqref{ResFam-Shift-even} and
\eqref{ResFam-Shift-odd} actually are zeros of the respective numerators. These
formulas naturally reflect the degrees of the families on both sides. In fact,
the degrees of the polynomials $D_{2N}^{res}(h;\lambda)$ and
$\S_{2N}(h;\lambda)$ are $N$ and $2N$, respectively. Similarly, the degrees of
$D_{2N+1}^{res}(h;\lambda)$ and $\S_{2N+1}(h;\lambda)$ are $N$ and $2N+1$,
respectively.

Corollary \ref{RFvsSF} also shows that the composition of the order $2N$ family
$S_N(g_+;\lambda)$ with $\iota^*$ degenerates to a family of order $N$. Some of
the properties of the families $S_N(g_+;\lambda)$ away from $r=0$ will be
discussed in Sections \ref{applications}--\ref{panorama}.

A direct consequence of Corollary \ref{RFvsSF} and Lemma \ref{CommShiftM} are
factorization identities for residue families $D_N^{res}(h;\lambda)$ into
compositions with the factors
$$
    S(g_+;\lambda+n-1) \quad \mbox{and} \quad M_r.
$$

\begin{cor}\label{Factor-residue} Let $N\in\N$ so that $2N \le n$ for even $n$. Then
\begin{align*}
   D_{2N}^{res}(h;\lambda) & = D_{2N-1}^{res}(h;\lambda-1) S(g_+;\lambda+n-1),\\
   -(2N+1)(2\lambda+n-2N-1) D_{2N+1}^{res}(h;\lambda) & = D_{2N}^{res}(h;\lambda-1)
   S(g_+;\lambda+n-1)
\end{align*}
and
\begin{align*}
   D_{2N}^{res}(h;\lambda) & = D_{2N+1}^{res}(h;\lambda+1) M_r,\\
   -2N(2\lambda+n-2N+2) D_{2N-1}^{res}(h;\lambda) & = D_{2N}^{res}(h;\lambda+1) M_r.
\end{align*}
\end{cor}

\begin{proof} The first two identities immediately follow from Corollary
\ref{RFvsSF}. The proofs of the last two identities also require Lemma
\ref{CommShiftM}. We omit the details.
\end{proof}

The following result is a consequence of the conformal transformation law in
Proposition \ref{ConformalCovarianceP}.

\begin{lem}\label{CTL-SN} Let $N \in \N$ and assume that $\hat{h}=e^{2\varphi}h$. Then
\begin{equation*}
   S_N(\hat{g}_+;\lambda) = \left(\frac{\kappa^*(r)}{r} \right)^{n-\lambda} \kappa^* S_N(g_+;\lambda)
   \kappa_* \left(\frac{\kappa^*(r)}{r} \right)^{\lambda+N-n}.
\end{equation*}
\end{lem}

\begin{proof} Proposition \ref{ConformalCovarianceP} yields
\begin{equation*}
   S(\hat{g}_+;\lambda) = \left(\frac{\kappa^*(r)}{r}\right)^{n-\lambda} \kappa^* S(g_+;\lambda)
   \kappa_* \left(\frac{\kappa^*(r)}{r} \right)^{\lambda-n+1}.
\end{equation*}
An application of that identity to the composition $S_N(\hat{g}_+;\lambda)$
gives
\begin{equation*}
   S_N(\hat{g}_+;\lambda) =\left(\frac{\kappa^*(r)}{r} \right)^{n-\lambda} \kappa^* S(g_+;\lambda)
   \cdots S(g_+;\lambda+N-1)
   \kappa_* \left(\frac{\kappa^*(r)}{r} \right)^{\lambda+N-n}.
\end{equation*}
The proof is complete.
\end{proof}

Now by combining Theorem \ref{DeltaN} with Lemma \ref{CTL-SN}, we obtain an
alternative proof of the conformal transformation law of residue families
\cite[Theorem $6.6.3$]{J1}.

\begin{cor}\label{ConFormalTrafoBSFamily} Let $N\in\N$ so that $N \le n+1$ for even $n$.
Assume that $\hat{h}=e^{2\varphi}h$. Then
\begin{equation*}
   D_N^{res}(\hat{h};\lambda) = e^{(\lambda-N)\varphi} D_N^{res}(h;\lambda) \kappa_*
   \left(\frac{\kappa^*(r)}{r} \right)^{\lambda}.
\end{equation*}
\end{cor}

\begin{proof} The assertion is a consequence of Theorem \ref{DeltaN}, Lemma \ref{CTL-SN}, the limit
formula \cite[(6.6.16)]{J1}
\begin{equation*}
   \lim_{r\to 0} \left(\frac{\kappa^*(r)}{r}\right)=e^{-\varphi}
\end{equation*}
and the fact that $\kappa$ acts as the identity on $M$.
\end{proof}

Finally, we prove

\begin{lem}\label{CommShiftM} Let $N\in \N$ and $f\in C^\infty(M_+^\circ)$. Then
\begin{equation}\label{Comm}
    S_N(g_+;\lambda)(r f) = r S_N(g_+;\lambda-1)(f) - N(2\lambda-n+N) S_{N-1}(g_+;\lambda)(f).
\end{equation}
\end{lem}

\begin{proof} We iteratively apply Lemma \ref{GeneralSL(2)} to compute
\begin{align*}
   S_N(g_+;\lambda)(r f) & = S(g_+;\lambda) \circ \cdots \circ S(g_+;\lambda+N-1)(r f)\\
   & = S_{N-1}(g_+;\lambda) r S(g_+;\lambda+N-2)(f) - (2\lambda-n+2N-1) S_{N-1}(g_+;\lambda)(f)\\
   & = \cdots = r S_N(g_+;\lambda-1)(f) - \left[N(2\lambda\!-\!n\!+\!2N\!+\!1) - \sum_{j=1}^N 2j\right]
   S_{N-1}(g_+;\lambda)(f)
\end{align*}
Now the identity $\sum_{j=1}^N 2j = N(N+1)$ completes the proof.
\end{proof}

The identity \eqref{Comm} obviously holds true also for all $f \in
C^\infty(M_+)$.

\section{Applications}\label{applications}

In the present section, we further exploit the relation between families of
shift operators and residue families. The flow of information will be in both
directions, i.e., we use facts on families of shift operators to derive
properties of residue families and also use properties of residue families to
derive properties of families of shift operators.

It was shown in \cite{J1,J2} that residue families $D_{2N}^{res}(h;\lambda)$
satisfy two systems of factorization identities. In Section \ref{recover} we
provide new proofs of these identities. They rest on the identification of two
special values of the families of shift operator in terms of GJMS operators.
These are given in Theorem \ref{BigGJMS} and Theorem \ref{BSOperatorVsGJMS},
respectively. Theorem \ref{BigGJMS} will be further exploited in Section
\ref{expansions}. In Section \ref{sol}, we shall describe a compressed
formulation of the recursive algorithm for the solution operators
$\T_{2j}(h;\lambda)$ in terms of the families $S_N(g_+;\lambda)$. In the last
section, we derive a new formula for all $Q$-curvatures (critical and
subcritical ones) in even dimension in terms of shift operators.

\subsection{Shift operators and GJMS operators}\label{recover}

In \cite{J1,J2}, it was proves that, for $N \in \N$ with $2N \le n$ for even
$n$, the even-order residue families $D_{2N}^{res}(h;\nu)$ satisfy the
identities
\begin{equation}\label{RF-Fact}
   D_{2N}^{res}\left(h;-\frac{n}{2}\!+\!N\right) = P_{2N}(h)\iota^* \quad
   \mbox{and} \quad
   D_{2N}^{res}\left(h;-\frac{n\!+\!1}{2}\!+\!N\right) = \iota^*P_{2N}(\bar{g}).
\end{equation}
We briefly comment on the well-definedness of the second identity. For odd $n$,
the Taylor coefficients of $h_r$ are determined by $h$. Hence, for any $N \in
\N$, the left-hand side of the second identity is determined by $h$. The
right-hand side of this identity involves a GJMS operator in even dimension
$n+1$. We recall that, for general metrics, these are only defined for
subcritical orders $2N \le n+1$. But here they are defined for all $N \in \N$.
In fact, there is an explicit formula for the Taylor coefficients of a
Poincar\'e metric of $\bar{g}$ \cite{J2}. These are determined by $h_r$, i.e.,
by $h$. Thus the right-hand side is well-defined for all $N \in \N$. For even
$n$, the left-hand side of the second identity is defined for $2N \le n$. The
right-hand side uses derivatives of a Poincar\'e metric for $\bar{g}$ which are
determined by $h$.

For even $N$, the identities \eqref{RF-Fact} are the simplest respective
special cases of the systems
\begin{equation}\label{factor-a}
   D_N^{res} \left(h;-\frac{n}{2}\!+\!N\!-\!k\right)
   = P_{2k}(h) D_{N-2k}^{res}\left(h;-\frac{n}{2}\!+\!N\!-\!k\right), \;\; 2 \le 2k \le N
\end{equation}
and
\begin{equation}\label{factor-b}
   D_N^{res}\left(h;-\frac{n\!+\!1}{2}\!+\!k\right)
   = D_{N-2k}^{res}\left(h;-\frac{n\!+\!1}{2}\!-\!k\right) P_{2k}(\bar{g}), \;\; 2 \le 2k \le N
\end{equation}
of factorization identities \cite[Theorems 3.1--3.2]{J2}. They play an
important role in connection with the description of recursive structures among
GJMS operators and $Q$-curvatures. Here we shall give an {\em independent
proof} of the second system. The arguments will also prove their counterparts
for odd-order residue families. Moreover, we derive the first system from its
special case $k=N$, i.e., from the first identity in \eqref{RF-Fact}. The new
proofs completely differ from earlier arguments.

We start with the proof of system \eqref{factor-b}. The proof rests on
Corollary \ref{RFvsSF} and two basic facts. The first of these is also of
independent interest. It will be used in Section \ref{expansions}.

\begin{theorem}\label{BigGJMS} Let $N \in \N$ so that $2N \le n$ if $n$ is even.
Set $m = \frac{n+1}{2}$. Then
\begin{equation*}
    S_N(g_+;\m-1) = r^N P_{2N}(\bar{g})
\end{equation*}
up to an error term in $O(r^\infty)$ for odd $n$ and $o(r^{n-N})$ for even $n$.
Moreover, the equality is true without an error term if $g_+$ is Einstein.
\end{theorem}

\begin{proof} We first consider the case $g_+$ Einstein. The identity
\eqref{eq:GJMSOnEinstein} shows that the $2N$-th order GJMS operator of $g_+$
is given by the product
\begin{equation}\label{eq:GJMSForg+}
    P_{2N}(g_+) = \prod_{l=1}^N \left(\Delta_{g_+}+ (\m\!+\!l\!-\!1)(m\!-\!l)\right).
\end{equation}
The conformal covariance of GJMS operators implies that
$$
   P_{2N}(\bar{g}) = r^{-\m-N}P_{2N}(g_+)r^{\m-N}.
$$
By the definition \eqref{eq:ThePpolynomial}, we have
\begin{equation*}
    S_N(g_+;\m-1) = S(g_+;\m-1) \circ \cdots \circ S(g_+;\m+N-2).
\end{equation*}
Hence, using \eqref{NewForm} and \eqref{eq:GJMSForg+}, we obtain
\begin{align*}
   S_N(g_+;\m-1) & = r^{-\m} \left(\Delta_{g_+} + \m(\m\!-\!1)\right)
   \cdots \left(\Delta_{g_+} + (\m\!+\!N\!-\!1)(\m\!-\!N)\right) r^{\m-N}\\
   & = r^{-\m} P_{2N}(g_+) r^{\m-N}\\
   & = r^N P_{2N}(\bar{g}).
\end{align*}
This completes the proof for $g_+$ Einstein. For general Poincar\'e metrics and
odd $n$, the assertion follows by similar arguments using the generalization
\begin{equation}\label{GJMS-prod-g-odd}
   P_{2N}(g_+) = \prod_{l=1}^N \left(\Delta_{g_+}+ (\m\!+\!l\!-\!1)(m\!-\!l)\right) +
   O(r^\infty)
\end{equation}
of \eqref{eq:GJMSForg+}. The conformal covariance of $P_{2N}$ show that
$$
   S_N(g_+;\m-1) = r^N P_{2N}(\bar{g}) + O(r^\infty).
$$
Similarly, for even $n$, the formula
\begin{equation}\label{GJMS-prod-g-even}
   P_{2N}(g_+) = \prod_{l=1}^N \left(\Delta_{g_+}+ (\m\!+\!l\!-\!1)(m\!-\!l)\right) + o(r^n)
\end{equation}
(see Remark \ref{AE} for $N=1$) and the conformal covariance of $P_{2N}$ show
that
$$
   S_N(g_+;\m-1) = r^N P_{2N}(\bar{g}) + o(r^{n-N}).
$$
The proof is complete.
\end{proof}

Theorem \ref{BigGJMS} directly implies the following $N$ factorization
identities.

\begin{cor}\label{IdentitiesForBSOperator} Let $N \in \N$ so that $2N \le n$ for even $n$.
Let $0 \leq k \leq N-1$. Then
\begin{equation*}
    S_N(g_+;\m-k-1) = S_k(g_+;\m-k-1) r^{N-k} P_{2N-2k}(\bar{g})
\end{equation*}
if $g_+$ is Einstein. For general Poincar\'e metrics $g_+$, the identity holds
true with an error term in $O(r^\infty)$ for odd $n$ and
$$
     S_k(g_+;\m-k-1) o(r^{n-N+k})
$$
for even $n$.
\end{cor}

The second basic fact is a generalization of Lemma \ref{GeneralSL(2)} which
states that
\begin{equation*}
    S(g_+;\lambda) r^j = r^j S(g_+;\lambda\!-\!j) - j(2\lambda\!-\!n\!-\!j\!+\!2) r^{j-1}
\end{equation*}
for $j \in \N$.

\begin{lem}\label{VG-SL(2)} Let $k,j\in\N$. Then
\begin{align}\label{eq:ExtendedSL}
    S_k(g_+;\lambda) r^j = \sum_{l=0}^k \binom{k}{l}(-j)_l (2\lambda\!-\!n\!-\!j\!+\!k\!+\!1)_l r^{j-l}
    S_{k-l}(g_+;\lambda\!-\!j\!+\!l).
\end{align}
\end{lem}

\begin{proof} We use induction over $k$. For $k=1$, we have
\begin{align*}
    S_1(g_+;\lambda)r^j & = S(g_+;\lambda) r^j = r^j S(g_+;\lambda\!-\!j) - j(2\lambda\!-\!n\!-\!j\!+\!2) r^{j-1}\\
    & = \sum_{l=0}^1 \binom{1}{l} (-j)_l (2\lambda\!-\!n\!-\!j\!+\!2)_l r^{j-l} S_{1-l}(g_+;\lambda\!-\!j\!+\!l).
\end{align*}
Now assume that the assertion holds true for $k-1$. Then we compute
\begin{align*}
    & S_k(g_+;\lambda)r^j \\ & = S(g_+;\lambda) \circ S_{k-1}(g_+;\lambda\!+\!1) r^j\\
    & = S(g_+;\lambda) \sum_{l=0}^{k-1} \binom{k-1}{l} (-j)_l (2\lambda\!-\!n\!-\!j\!+\!k\!+\!2)_l r^{j-l}
    S_{k-l-1}(g_+;\lambda\!-\!j\!+\!l\!+\!1)\\
    & = \sum_{l=0}^{k-1} \binom{k-1}{l}(-j)_l (2\lambda\!-\!n\!-\!j\!+\!k+\!2)_l S(g_+;\lambda) r^{j-l}
    S_{k-l-1}(g_+;\lambda\!-\!j\!+\!l\!+\!1).
\end{align*}
By Lemma \ref{GeneralSL(2)}, the last display equals
\begin{align*}
    & \sum_{l=0}^{k-1} \binom{k-1}{l} (-j)_l (2\lambda\!-\!n\!-\!j\!+\!k\!+\!2)_l r^{j-l}
    S(g_+;\lambda\!-\!j\!+\!l) S_{k-l-1}(g_+;\lambda\!-\!j\!+\!l\!+\!1)\\
    & - \sum_{l=0}^{k-1} \binom{k-1}{l} (-j)_l (2\lambda\!-\!n\!-\!j\!+\!k+\!2)_l \\
    & \quad \times (j-l)(2\lambda\!-\!n\!-\!j\!+\!l\!+\!2) r^{j-l-1} S_{k-l-1}(g_+;\lambda\!-\!j+\!l+\!1).
\end{align*}
Shifting the summation index in the second sum and simplification gives
\begin{align*}
    & \sum_{l=0}^{k-1} \binom{k-1}{l} (-j)_l (2\lambda\!-\!n\!-\!j\!+\!k+\!2)_l r^{j-l}
    S_{k-l}(g_+;\lambda\!-\!j\!+\!l)\\
    & + \sum_{l=1}^{k} \binom{k-1}{l-1} (-j)_{l} (2\lambda\!-\!n\!-\!j\!+\!k\!+\!2)_{l-1}
    (2\lambda\!-\!n\!-\!j\!+\!l\!+\!1) r^{j-l} S_{k-l}(g_+;\lambda\!-\!j\!+\!l).
\end{align*}
Now observe that
\begin{equation*}
    (2\lambda\!-\!n\!-\!j\!+\!k\!+\!2)_l = (2\lambda\!-\!n\!-\!j\!+\!k\!+\!1)_l
    + l(2\lambda\!-\!n\!-\!j\!+\!k\!+\!2)_{l-1}
\end{equation*}
and
\begin{multline*}
    (2\lambda-n-j+k+2)_{l-1}(2\lambda-n-j+l+1) = (2\lambda-n-j+k+1)_l-(k-l)(2\lambda-n-j+k+2)_{l-1}
\end{multline*}
as well as
\begin{equation*}
    l \binom{k-1}{l}-(k-l)\binom{k-1}{l-1} =0 \quad \mbox{and} \quad \binom{k-1}{l}-\binom{k-1}{l-1} =\binom{k}{l}.
\end{equation*}
Putting things together, we conclude
\begin{equation*}
    S_k(g_+;\lambda)r^j = \sum_{l=0}^k \binom{k}{l}(-j)_l (2\lambda\!-\!n\!-\!j\!+\!k\!+\!1)_l r^{j-l}
    S_{k-l}(g_+;\lambda\!-\!j\!+\!l).
\end{equation*}
The proof is complete.
\end{proof}

Now we are able to prove the system \eqref{factor-b}.

\begin{theorem}\label{second-np} Let $N\in\N$ so that $N \le n+1$ for even $n$. Then
\begin{equation}\label{factor-2}
    D_N^{res}\left(h;-\frac{n\!+\!1}{2}\!+\!k\right)
    = D_{N-2k}^{res}\left(h;-\frac{n\!+\!1}{2}\!-\!k\right) P_{2k}(\bar{g})
\end{equation}
for all $0\leq 2k\leq N$.
\end{theorem}

\begin{proof} We only discuss the case of even $N$. The proof in the odd-order case
is analogous. By Corollary \ref{RFvsSF}, we have
\begin{equation*}
    D_{2N}^{res}\left(h;-\frac{n\!+\!1}{2}\!+\!k\right) = \frac{1}{(-2N)_N(k-2N)_N} \iota^*
    S_{2N}\left(g_+;\frac{n\!-\!1}{2}\!+\!k\!-\!2N\right).
\end{equation*}
Now we additionally assume that $g_+$ is Einstein. Then Corollary
\ref{IdentitiesForBSOperator} implies the identity
\begin{equation}\label{inter}
    S_{2N}\left(g_+;\frac{n\!-\!1}{2}\!+\!k\!-\!2N\right)
    = S_{2N-k}\left(g_+;\frac{n\!-\!1}{2}\!+\!k\!-\!2N\right) r^k P_{2k}(\bar{g})
\end{equation}
We use Lemma \ref{VG-SL(2)} to conclude that
\begin{multline*}
    S_{2N-k}\left(g_+;\frac{n\!-\!1}{2}\!+\!k\!-\!2N\right)r^k \\
    = \sum_{l=0}^{2N-k} \binom{2N-k}{l} (-k)_l(-2N)_l r^{k-l}
    S_{2N-k-l}\left(g_+;\frac{n\!-\!1}{2}\!+\!l\!-\!2N\right).
\end{multline*}
Note that no summand with summation index $l\geq k+1$ will contribute due to
$(-k)_l=0$. By restriction to $r=0$, the last display yields
\begin{multline}\label{inter2}
    \iota^*S_{2N-k}\left(g_+;\frac{n\!-\!1}{2}\!+\!k\!-\!2N\right) r^k \\
    = \binom{2N-k}{k}(-k)_k(-2N)_k \iota^* S_{2N-2k}\left(g_+;\frac{n\!-\!1}{2}\!+\!k\!-\!2N\right).
\end{multline}
Finally, by Corollary \ref{RFvsSF}, we have
\begin{align*}
    D_{2N-2k}^{res}\left(h;-\frac{n\!+\!1}{2}\!-\!k\right)
    = \frac{1}{(-2N\!+\!2k)_{N-k} (k\!-\!2N)_{N-k}} \iota^*
    S_{2N-2k}\left(g_+;\frac{n\!-\!1}{2}\!+\!k\!-\!2N\right).
\end{align*}
Combining these observation, proves
\begin{align*}
    D_{2N}^{res}\left(h;-\frac{n\!+\!1}{2}\!+\!k\right)
    & = \frac{\binom{2N-k}{k} (-k)_k (-2N)_k (-2N\!+\!2k)_{N-k} (k\!-\!2N)_{N-k}}{(-2N)_N(k-2N)_N} \\
    & \times D_{2N-2k}^{res}\left(h;-\frac{n\!+\!1}{2}\!-\!k\right) P_{2k}(\bar{g}).
\end{align*}
Now the combinatorial identity
\begin{align*}
    \frac{\binom{2N-k}{k}(-k)_k(-2N)_k(-2N+2k)_{N-k}(k-2N)_{N-k}}{(-2N)_N(k-2N)_N} = 1
\end{align*}
completes the proof for Einstein $g_+$. For general Poincar\'e metrics $g_+$,
we have to control the error terms coming from Corollary
\ref{IdentitiesForBSOperator}. For odd $n$, error terms obviously do to not
contribute to \eqref{inter2}. For even $n$, the relation \eqref{inter} contains
an error term in
$$
   S_{2N-k}\left(g_+;\frac{n\!-\!1}{2}\!+\!k\!-\!2N\right) o(r^{n-k}).
$$
By Lemma \ref{VG-SL(2)}, this contribution is contained in $o(r^{n-2N})$. Hence
its composition with $\iota^*$ vanishes if $2N \le n$.\footnote{The arguments
show that, for even $n$ and $2N < n$, the weaker estimate $O(r^n)$ in
\eqref{GJMS-prod-g-even} suffices. However, the critical case $2N=n$ requires
the stronger estimates $o(r^n)$.}
\end{proof}

\begin{remark} The proof of the factorizations \eqref{factor-2} for even $N$
given in \cite{J2} assumes that $g_+$ is Einstein. A closer inspection of this
proof shows that it can be refined to establish the assertion in full
generality. The refinement rests on the factorizations \eqref{GJMS-prod-g-odd}
and \eqref{GJMS-prod-g-even} with remainder terms. The point is that the
remainder terms do not contribute to the residue calculations in the refinement
of that proof. Note also that, along these lines, again only the critical case
$2N=n$ (for $n$ even) requires the remainder term $o(r^n)$ in
\eqref{GJMS-prod-g-even}.
\end{remark}

We continue with the discussion of the factorization identities
\eqref{factor-a}. Their proof rests on Corollary \ref{RFvsSF} and the
identification of the value
$$
   \iota^*S_{2N}\left(g_+;\frac{n}{2}\!-\!N\right) = \S_{2N}\left(h,\frac{n}{2}\!-\!N\right)
$$
as a tangential operator being a multiple of $P_{2N}(h) \iota^*$ . This fact
actually will be deduced from the first identity in \eqref{RF-Fact}.
Unfortunately, we do not have an independent proof of this identity. The
following result contains the relevant details and some further information.

\begin{theorem}\label{BSOperatorVsGJMS} Let $k\in\N$. The operator $\iota^*
S_k(g_+;\frac{n-k}{2})$ defines a tangential differential operator $\P_k:
C^\infty(M) \to C^\infty(M)$, i.e.,
\begin{equation*}
    \iota^* S_{k}\left(g_+;\frac{n-k}{2}\right) = \P_k \iota^*.
\end{equation*}
For $k \in \N$ with $k \le n$ for even $n$, the operator $\P_k$ only depends on
$h$ and is conformally covariant, i.e.,
\begin{equation*}
   \P_{k}(\hat{h})=e^{-(\frac{n+k}{2})\varphi} \circ \P_{k}(h) \circ e^{(\frac{n-k}{2})\varphi}
\end{equation*}
for $\hat{h}=e^{2\varphi}h$. For $k=2N-1$, the operator $\P_k(h)$ vanishes
identically. For $k=2N$, the operator $\P_k(h)$ is proportional to the GJMS
operator $P_{2N}(h)$ of $(M,h)$:
\begin{equation}\label{identify}
   \P_{2N}(h) = ((2N\!-\!1)!!)^2 P_{2N}(h).
\end{equation}
Finally, we have
\begin{equation}\label{Seven-special}
    \iota^* S_{2N}\left(g_+;\frac{n\!-\!1}{2}\!-\!N\right) = (2N)! \iota^* P_{2N}(\bar{g})
\end{equation}
and
\begin{equation}\label{Sodd-special}
    \iota^* S_{2N+1}\left(g_+;\frac{n\!-\!3}{2}\!-\!N\right) = (2N+2)! \iota^* \partial_r
    P_{2N}(\bar{g}).
\end{equation}
\end{theorem}

\begin{proof} We recall that a differential operator $D: C^\infty(M_+) \to
C^\infty(M_+)$ restricts to a tangential operator with respect to $M$ if and
only if $D(r f) = r D^\prime(f)$ for all $f\in C^\infty(M_+)$ and some
differential operator $D^\prime: C^\infty(M_+) \to C^\infty(M_+)$. By Lemma
\ref{CommShiftM}, we have
\begin{equation*}
    S_k(g_+;\lambda)(r f) = r S_k(g_+;\lambda-1)f - k(2\lambda-n+k) S_{k-1}(g_+;\lambda-1)f
\end{equation*}
for all $f \in C^\infty(M_+)$. Hence the Taylor series of $\iota^*
S_{k}(g_+;\frac{n-k}{2})(r f)$ in the variable $r$ has vanishing constant term.
It follows that $\iota^* S_k(g_+;\frac{n-k}{2})$ defines a tangential operator,
i.e., there is an operator
\begin{equation*}
   \P_k: C^\infty(M) \to C^\infty(M)
\end{equation*}
so that
$$
  \iota^* S_k \left(g_+;\frac{n-k}{2}\right) = \P_k \iota^*.
$$
Now assume that $k \in \N$ with $k \le n$ for even $n$. The further properties
of $\P_k$ follow from the relation between shift operators and residue families
(Corollary \ref{RFvsSF}). In particular, for these values of $k$, the operators
$\P_k$ are determined by $h$ and the conformal transformation law for $\P_k$
follows from Corollary \ref{ConFormalTrafoBSFamily}. The fact that $\P_k$
vanishes identically for odd $k=2N-1$ is obvious. Indeed, by Corollary
\ref{RFvsSF}, we have
\begin{align*}
    \P_{2N-1} & =\iota^* S_{2N-1}\left(g_+;\frac{n+1}{2}\!-\!N\right) \\
    & = 2(-2N\!+\!1)_N (-N\!+\!1)_N D^{res}_{2N-1}\left(h;N\!-\!\frac{n+1}{2}\right) = 0
\end{align*}
since $(-N\!+\!1)_N =0$ and $D^{res}_{2N-1}(h;\lambda)$ is regular in
$\lambda$. By Corollary \ref{RFvsSF} and the first factorization relation in
\eqref{RF-Fact}, we conclude that
\begin{align*}
    \P_{2N} \iota^* & = \iota^* S_{2N}\left(g_+;\frac{n}{2}\!-\!N\right) \\
    & =(-2N)_N \left(-N\!+\!\frac{1}{2}\right)_N D_{2N}^{res}\left(h;N\!-\!\frac{n}{2}\right) \\
    & =((2N\!-\!1)!!)^2 P_{2N}(h) \iota^*.
\end{align*}
The identity \eqref{Seven-special} follows from Corollary \ref{RFvsSF} and the
second factorization identity in \eqref{RF-Fact}. Similarly, the identity
\eqref{Sodd-special} follows from the factorization identity
$$
   D_{2N+1}^{res}\left(h;-\frac{n\!+\!1}{2}\!+\!N\right) =
   D_1^{res}\left(h;-\frac{n\!+\!1}{2}\!-\!N\right) P_{2N}(\bar{g})
$$
for odd-order residue families (Theorem \ref{second-np}). By Corollary
\ref{RFvsSF}, this relation is equivalent to
\begin{align*}
   & \frac{1}{2(-2N\!-\!1)_{N+1} (-N\!-\!1)_{N+1}} \iota^*
   S_{2N+1}\left(g_+;-N\!+\!\frac{n\!-\!3}{2}\right) \\ & = \frac{1}{2(N\!+\!1)} \iota^*
   S_1\left(g_+;-N\!+\!\frac{n\!-\!3}{2}\right) P_{2N}(\bar{g}).
\end{align*}
Simplification proves the claim. The proof is complete.
\end{proof}

Some comments on the latter results are in order.

Theorem \ref{BSOperatorVsGJMS} overlaps with \cite[Theorems 4.1 and 4.5]{GW}.
Indeed, assume that $g_+$ is Einstein. Then, by Proposition
\ref{EquivalentOperators}, $S_k(g_+;\lambda)$ can be regarded as a composition
of $k$ degenerate Laplacians $\IDD$. The fact that $\iota^*
S_k(g_+;\frac{n-k}{2})$ is tangential, is a consequence of a basic
$sl(2)$-structure for the degenerate Laplacian \cite[Section 3.1]{GW}. It holds
true for compositions of degenerate Laplacians in a much more general setting.
In the present situation, Lemma \ref{CommShiftM} and Lemma \ref{VG-SL(2)}
reflect that structure. In order to relate the tangential operator $\P_{2N}$ to
the GJMS operator $P_{2N}(h)$, we used the first relation in \eqref{RF-Fact}.
Note that this relation is a consequence of the basic residue relation
\eqref{Res-SO} (derived in \cite{GZ} from the ambient metric construction of
$P_{2N}(h)$). In \cite{GW}, the identification \eqref{identify} also rests on
ambient metric arguments.

As a consequence of Theorem \ref{BSOperatorVsGJMS}, we obtain a new proof of
the first system \eqref{factor-a} of factorization identities for residue
families.

\begin{cor}\label{res-factor} Let $N \in \N$ with $N \le n+1$ for even $n$.
Then we have the factorization identities
\begin{equation}\label{factor-res}
   D_N^{res} \left(h;-\frac{n}{2}\!+\!N\!-\!k\right)
   = P_{2k}(h) D_{N-2k}^{res}\left(h;-\frac{n}{2}\!+\!N\!-\!k\right)
\end{equation}
for $0 \le 2k \le N$.
\end{cor}

\begin{proof} The obvious relation
$$
   S_N(g_+;\lambda) = S_{2k}(g_+;\lambda) S_{N-2k}(g_+;\lambda+2k)
$$
implies the identity
$$
   \iota^* S_N\left(g_+;\frac{n}{2}\!-\!k\right) = \iota^*
   S_{2k}\left(g_+;\frac{n}{2}-k\right) S_{N-2k}\left(g_+;\frac{n}{2}\!+\!2k\right).
$$
By Theorem \ref{BSOperatorVsGJMS}, it is equivalent to
$$
   \iota^* S_N\left(g_+;\frac{n}{2}\!-\!k\right) = ((2k\!-\!1)!!)^2 P_{2k}(h) \iota^*
   S_{N-2k}\left(g_+;\frac{n}{2}\!+\!2k\right).
$$
Now the identity \eqref{ResFam-Shift-even} shows that, for even $N$, this
relation can be restated as
\begin{multline*}
   (-2N)_N \left(-k\!+\!\frac{1}{2}\right)_N D_{2N}^{res}\left(h;-\frac{n}{2}\!+\!2N\!-\!k\right) \\
   = ((2k\!-\!1)!!)^2 (-2N\!+\!2k)_{N-k} \left(k\!+\!\frac{1}{2}\right)_{N-k}
   P_{2k}(h) D_{2N-2k}^{res}\left(h;-\frac{n}{2}\!+\!2N\!-\!k\right).
\end{multline*}
Simplification proves the claim for even-order residues families. We omit the
analogous proof for odd-order residue families which utilizes the identity
\eqref{ResFam-Shift-odd}.
\end{proof}

Since the family $S_N(g_+;\lambda)$ is a polynomial of degree $N$ in $\lambda$,
the $N$ identities in Corollary \ref{IdentitiesForBSOperator} do {\em not}
suffice to determine $S_N(g_+;\lambda)$ in terms of lower-order
$S_k(g_+;\lambda)$ and GJMS operators of $\bar{g}$. However, that is possible
by combining Corollary \ref{IdentitiesForBSOperator} with the following formula
for the leading coefficient of that polynomial. We also recall that $w(r) =
\sqrt{v(r)}$.

\begin{prop}\label{LeadingTermForBSOperator} Let $N\in\N$. Then
\begin{equation}\label{LTS}
    \frac{1}{N!}\frac{d^N}{d \lambda^N} S_N(g_+;\lambda) = (-2)^N  w^{-1} \partial_r^N(w \cdot).
\end{equation}
\end{prop}

\begin{proof} The relations \eqref{v-trace} and
$$
    w^{-1}\partial_r(w f) = \frac{1}{2} v^{-1} \partial_r(v) f + \partial_r f, \quad f\in C^\infty(M_+^\circ)
$$
show that we can rewrite the operator $S(g_+;\lambda)$ as
\begin{equation*}
    S(g_+;\lambda) = r\Delta_{\bar{g}} + (n-1)(\partial_r+ v^{-1} \partial_r (v)) - 2\lambda w^{-1}\partial_r(w\cdot).
\end{equation*}
Now the leading coefficient of the polynomial $\lambda \to S_N(g_+;\lambda)$
coincides with the product of the leading coefficients of the $N$ factors. But
these are all given by $-2 w^{-1}\partial_r(w\cdot)$. This proves the
assertion.
\end{proof}

For later use, we introduce the notation $\partial_r^w \st
w^{-1}\partial_r(w\cdot)$. In these terms, the right-hand side of \eqref{LTS}
equals $(-2)^N (\partial_r^w)^N$. In Section \ref{expansions}, we shall discuss
further consequences of Corollary \ref{IdentitiesForBSOperator} and Proposition
\ref{LeadingTermForBSOperator}.

Finally, we combine Proposition \ref{LeadingTermForBSOperator} with Corollary
\ref{RFvsSF} to read off the leading coefficients of residue families. The
definition of residue families implies formulas for these coefficients in terms
of solution operators $\T_{2j}(h;\lambda)$ and renormalized volume coefficients
$v_{2j}$. The following result shows how these can be simplified.

\begin{cor}\label{leading} Let $N \in \N$ with $2N \le n$ for even $n$.
The leading coefficient of the even-order residue family $\lambda \mapsto
D_{2N}^{res}(h;\lambda)$ equals
$$
   (-1)^N 2^{2N} \frac{N!}{(2N)!} \iota^* \partial_r^{2N} (w \cdot).
$$
Similarly, the leading coefficient of the odd-order residue family $\lambda
\mapsto D_{2N+1}^{res}(h;\lambda)$ equals
$$
   (-1)^N 2^{2N} \frac{N!}{(2N\!+\!1)!} \iota^* \partial_r^{2N+1} (w \cdot).
$$
\end{cor}

\begin{proof} The even-order residue family
$D_{2N}^{res}(h;\lambda)$ is a polynomial of degree $N$. By Proposition
\ref{LeadingTermForBSOperator}, Corollary \ref{RFvsSF} and $\iota^* w(r) = 1$,
the coefficient of $\lambda^{N}$ equals
$$
   (-2)^{2N} \frac{1}{(-2N)_N} \iota^* \partial_r^{2N}(w \cdot).
$$
This proves the assertion. The proof in the odd-order case is analogous.
\end{proof}

\subsection{Shift operators and solution operators}\label{sol}

In the present section, we assume that $g_+$ is Einstein and that $n$ is odd.
By the latter assumption, all Taylor coefficients of $h_r$ are determined by
$h$. The obvious modifications for even $n$ are left to the reader. We relate
the solution operator $\T_{2N}(h;\lambda)$ (see Section \ref{setting}) to the
coefficients in the formal power series
\begin{equation}\label{eq:ExpansionOfP}
   S(g_+;\lambda) = -(2\lambda\!-\!n\!+\!1) \partial_r + r \sum_{k\geq 0} r^k S^{(k)}(h;\lambda)
\end{equation}
following from the formal power series
\begin{equation*}
   \Delta_{\bar{g}} + (\lambda\!-\!n\!+\!1)\J(\bar{g}) = \sum_{k\geq 0} r^k S^{(k)}(h;\lambda).
\end{equation*}
Here we used the identity \eqref{Jbar}.\footnote{The assumptions guarantee that
\eqref{Jbar} is an identity of formal power series.} Only the operator
$S^{(0)}(h;\lambda)$ contains two derivatives in $r$. By \eqref{eq:CCLaplacian}
and \cite[Lemma $6.11$]{J1}, the first few coefficients in the expansion
\eqref{eq:ExpansionOfP} are given by the operators
\begin{align}\label{eq:PCoefficients}
   S^{(0)}(h;\lambda)f & = \Delta_h f + \partial_r^2 f + (\lambda\!-\!n\!+\!1)\J(h)f,\notag\\
   S^{(1)}(h;\lambda)f & = -\J(h) \partial_r f, \notag\\
   S^{(2)}(h;\lambda)f & = -\delta_h(\Rho(h)\# d f) -\frac{1}{2} h(d\J(h),df) +
   \frac{1}{2}(\lambda\!-\!n\!+\!1) |\Rho(h)|^2 f, \notag\\
   S^{(3)}(h;\lambda)f & = -\frac{1}{2} |\Rho(h)|^2\partial_r f.
\end{align}
We recall that, under the present assumptions, the solution operators
$\T_{2N}(h;\lambda)$ are well-defined for all $N \in \N$. In the following, we
shall regard functions in $C^\infty(M)$ as functions on $M_+$ that do not
depend on $r$, i.e., $\partial_r$ annihilates functions in $C^\infty(M)$.

\begin{prop}\label{GJMSVsBS} Let $N \in\N$. Then
\begin{equation}\label{algo}
   -2N(2\lambda\!-\!n\!+\!2N) \T_{2N}(h;\lambda)
   = \sum_{k=0}^{N-1} S^{(2N-2k-2)}(h;n\!-\!\lambda\!-\!2k\!-\!1) \T_{2k}(h;\lambda)
\end{equation}
as an identity of operators acting on $C^\infty(M)$. In particular, the $2N$-th
GJMS operator on $(M^n,h)$ is given by
\begin{equation}\label{eq:GJMSInTermsOfBS}
   P_{2N}(h) = -2^{2N-2}((N\!-\!1)!)^2 \sum_{k=0}^{N-1} S^{(2N-2k-2)}\left(h;\frac{n}{2}\!+\!N\!-\!2k\!-\!1\right)
   \T_{2k}\left(h;\frac{n}{2}\!-\!N\right).
\end{equation}
\end{prop}

\begin{proof} In \eqref{eq:TheDistribution}, we replace the eigenfunction $u$ by its
asymptotic expansion
$$
   \sum_{j\ge 0} r^{\nu +2j} a_{2j}(h;\nu) + \sum_{j\ge 0} r^{n-\nu+2j}
   b_{2j}(h;\nu)
$$
(see \eqref{ansatz}). In order to simplify the following equations, we shall
suppress the second sum. Then Theorem \ref{CurvedBS} implies
\begin{equation*}
   S(g_+;\lambda) \left(\sum_{j\geq 0}r^{\lambda+\nu-n+2j+1} a_{2j}(h;\nu)
   \right) =(\lambda\!+\!\nu\!-\!n\!+\!1)(\nu\!-\!\lambda\!-\!1)
   \left(\sum_{j \geq 0} r^{\lambda+\nu-n+2j} a_{2j}(h;\lambda)\right).
\end{equation*}
By Lemma \ref{GeneralSL(2)}, this relation is equivalent to
\begin{equation*}
   \sum_{j\geq 0} r^{\lambda+\nu-n+2j+1} S(g_+;n\!-\!\nu\!-\!2j\!-\!1) a_{2j}(h;\nu)
   = -\sum_{j\geq 0} 2j(2\nu\!-\!n\!+\!2j) r^{\lambda+\nu-n+2j} a_{2j}(h;\nu).
\end{equation*}
We rewrite this identity in terms of the expansion \eqref{eq:ExpansionOfP} and
compare coefficients of powers of $r$. This gives
\begin{multline*}
   S^{(2j-2)}(h;n\!-\!\nu\!-\!1) a_0(h;\nu) + \cdots + S^{(0)}(h;n\!-\!\nu\!-\!2j\!+\!1) a_{2j-2}(h;\nu)\\
   =-2j(2\nu\!-\!n\!+\!2j) a_{2j}(h;\nu)
\end{multline*}
for $j\geq 1$. Using $a_{2k}(h;\nu) = \T_{2k}(h;\nu) a_0(h;\nu)$, we obtain
\begin{equation*}
   -2j(2\nu\!-\!n\!+\!2j) \T_{2j}(h;\nu)
   = \sum_{k=0}^{j-1} S^{(2j-2k-2)}(h;n\!-\!\nu\!-\!2k\!-\!1) \T_{2k}(h;\nu).
\end{equation*}
Note that $\T_{2j}(h;\nu)$ has a simple pole at $\nu=\frac n2-j$ with residue
given by $P_{2j}(h)$ (see \eqref{Res-SO}). Thus the last display implies
\begin{equation*}
   P_{2j}(h) = -2^{2j-2}((j\!-\!1)!)^2 \sum_{k=0}^{j-1} S^{(2j-2k-2)}
   \left(h;\frac{n}{2}\!+\!j\!-\!2k\!-\!1\right) \T_{2k}\left(h;\frac{n}{2}\!-\!j\right).
\end{equation*}
The proof is complete.
\end{proof}

Proposition \ref{GJMSVsBS} is a compressed version of the usual algorithm for
the calculation of the solution operators. We shall illustrate it by low-order
examples in Section \ref{panorama}.

\subsection{Holographic formulas for $Q$-curvatures}\label{Q}

In the present section, we assume that $n$ is even. We shall discuss new
holographic formulas for the $Q$-curvatures $Q_{2N}(h)$ for $2N \le n$.

We start with a simple proof of a result which is also of independent interest
(see \cite[Theorem 1.6.6]{BJ}). It has been useful in connection with a
discussion of the recursive structure of $Q$-curvatures \cite{J4}.

\begin{prop}\label{vanish} Assume that $n$ is even and let $N \in \N$ with $2N \le n$.
Then
$$
   D_{2N}^{res}(h;0)(1) = 0.
$$
\end{prop}

\begin{proof} Corollary \ref{RFvsSF} implies that
$$
   D_{2N}^{res}(h;0)(1) = \frac{1}{(-2N)_N (\frac{n+1}{2}-2N)_N} \S_{2N}(h;n-2N)(1).
$$
By definition, we have
$$
   S_{2N}(g_+;n-2N) = S(g_+;n-2N) \circ \dots \circ S(g_+;n-1).
$$
But \eqref{eq:BSOperators} shows that
\begin{equation}\label{S-vanish}
   S(g_+;n-1)(1) = 0.
\end{equation}
Hence $\S_{2N}(h;n-2N)(1)=0$. This completes the proof.
\end{proof}

The polynomial $\lambda \mapsto D_{2N}^{res}(h;\lambda)(1)$ is called the
$Q$-curvature polynomial.

The $Q$-curvature polynomial $D_{2N}^{res}(h;\lambda)(1)$ also vanishes in odd
dimensions $n$. The above arguments prove this fact if $(\frac{n+1}{2}-2N)_N
\ne 0$.\footnote{In the remaining cases, the polynomial $\S_{2N}(h;\lambda)(1)$
has a double zero at $\lambda = n-2N$.}

In the critical case $2N=n$, Proposition \ref{vanish} states that
$D_{n}^{res}(h;0)(1) = 0$. Of course, this result also follows from
$D_n^{res}(h;0) = P_n(h) \iota^*$ (see \eqref{RF-Fact}) and $P_n(h)(1)=0$.

We continue with the discussion of the critical $Q$-curvature $Q_n(h)$ and
recall the holographic formula \cite[Theorem 6.6.1]{J1}
\begin{equation}\label{Q-holo}
   Q_n(h) = -(-1)^{\frac{n}{2}} \dot{D}_n^{res}(h;0)(1).
\end{equation}
The following result is a consequence of this identity.

\begin{theorem}[\bf Holographic formula for critical $Q$-curvature]\label{Q-holo-new}
Let $n$ be even. Then
\begin{equation}\label{Q-D}
   Q_n(h) = (-1)^\frac{n}{2} D_{n-1}^{res}(h;-1) \partial_r (\log v)
\end{equation}
or, equivalently,
\begin{equation}\label{Q-S}
   Q_n(h) = c_n \S_{n-1}(h;0) \partial_r (\log v)
\end{equation}
with $c_n = (-1)^{\frac{n}{2}} 2^{n-2} (\Gamma(\frac{n}{2})/\Gamma(n))^2$.
\end{theorem}

We recall that the composition $\S_{n-1}(h;0) = \iota^* S_{n-1}(g_+;0)$ only
depends on $h$.

\begin{proof} The critical special case $2N=n$ of the first factorization identity in
Corollary \ref{Factor-residue} reads
$$
   D_n^{res}(h;\lambda) = D_{n-1}^{res}(h;\lambda-1) S(g_+;\lambda+n-1).
$$
Now we differentiate this relation at $\lambda=0$ and use the vanishing result
\eqref{S-vanish}. We obtain
$$
   \dot{D}_n^{res}(h;0)(1) = D_{n-1}^{res}(h;-1) \dot{S}(g_+;n-1)(1)
$$
But Proposition \ref{LeadingTermForBSOperator} shows that $\dot{S}(g_+;n-1) =
-2 w^{-1} \partial_r (w \cdot)$. Hence $\dot{S}(g_+;n-1)(1) = -2 w^{-1}
\partial_r (w) = - 2 \partial_r (\log w) = -\partial_r (\log v)$. The relation \eqref{Q-D}
follows by combining these facts with the holographic formula \eqref{Q-holo}.
The second relation \eqref{Q-S} follows by combining the first relation with
Corollary \ref{RFvsSF}.
\end{proof}

\begin{remark} Formula \eqref{Q-S} should be compared with the special case
\begin{equation}\label{GW-Q}
   ((n\!-\!1)!!)^2 Q_n(h) = \iota^* \ID_{n-1}[-n\!+\!1] \circ \ID_L[1] \log (1)
\end{equation}
of \cite[Theorem 4.7]{GW}, where we set $\ID[\lambda] = \ID[\bar{g},r;\lambda]$ and
likewise for $\ID_L$. Here the factor $\ID_L$ is defined to act on the log density
$\log(\mu)$ ($\mu$ any positive smooth function) according to
$$
  \ID_L[g,\sigma;\omega] \log (\mu) \st [-\sigma \Delta_g + (n\!-\!1) g(d\sigma,d\cdot)]
  \log(\mu) - \frac{\omega}{n\!+\!1} \left[(n\!-\!1) \Delta_g (\sigma) + 2n \sigma
  \J(g)\right]
$$
(see \cite[Section 2]{GW}). For $g = \bar{g}$, $\sigma = r$ and $\mu = 1$, we find
\begin{align}\label{ID-special}
   \ID_L[\bar{g},r;\omega] \log (1) &
   = - \frac{\omega}{n\!+\!1} \left[(n\!-\!1) \Delta_{\bar{g}}(r) + 2nr \J(\bar{g})\right] \notag \\
   & = - \frac{\omega}{n\!+\!1} \left[\frac{n\!-\!1}{2} \tr (h_r^{-1} \dot{h}_r) - n \tr (h_r^{-1}
   \dot{h}_r) \right] \notag \\
   & = \frac{\omega}{2} \tr (h_r^{-1} \dot{h}_r)
   = \omega \partial_r(\log v).
\end{align}
Hence the formula \eqref{GW-Q} reduces to \eqref{Q-S}, up to a sign due to
conventions. Now the transformation law $\mu^n Q_n(\mu^2 h) = P_n(h) \log (\mu) +
Q_n(h)$ and \eqref{identify} imply
\begin{multline*}
   ((n\!-\!1)!!)^2 \mu^n Q_n(\mu^2 h) \\ = \iota^* \ID_{n-1}[-n\!+\!1] \circ \ID[0] \log(\mu)
   + \iota^* \ID_{n-1}[-n\!+\!1] \circ \ID_L[1] \log (1).
\end{multline*}
But the relation
$$
   \ID_L[1] a - \ID_L[1] b = \ID[0](a-b)
$$
yields
$$
   \ID[0] \log(\mu) + \ID_L[1] \log (1) = \ID_L[1] \log(\mu).
$$
Hence
$$
   ((n\!-\!1)!!)^2 \mu^n Q_n(\mu^2 h) = \iota^* \ID_{n-1}[-n\!+\!1] \ID_L[1] \log(\mu).
$$
This proves \cite[Theorem 4.7]{GW}.\footnote{Strictly speaking, $\mu$ on the
right-hand side is a function on $M_+$ and the formula for $Q_n$ does only depend on
the restriction of $\mu$ to the hypersurface $M$.}
\end{remark}

The formula \eqref{Q-S} resembles the well-known beautiful formula
$$
    Q_n(h) = (-1)^{\frac{n}{2}-1} {\bf \Delta}^{\frac{n}{2}}(\log t)|_{\rho=0, t=1}
$$
of Fefferman and Hirachi \cite{FH}. Here ${\bf \Delta}$ denotes the Laplacian
of the ambient metric in normal form relative to $h$, and $t$ is the
homogeneous coordinate on the ambient space $\R^+ \times M \times
(-\varepsilon,\varepsilon)$ with coordinates $(t,x,\rho)$.

Next, we establish a generalization of Theorem \ref{Q-holo-new} to all
subcritical $Q$-curvatures.

\begin{theorem}[\bf Holographic formula for subcritical $Q$-curvatures]\label{Q-holo-g}
Let $n$ be even and assume that $2N < n$. Then
\begin{equation}\label{QHG}
    Q_{2N}(h) = c_{2N} \S_{2N-1}\left(h;\frac{n}{2}-N\right) \partial_r (\log v),
\end{equation}
where $c_{2N} = (-1)^{N} 2^{2N-2} (\Gamma(N)/\Gamma(2N))^2$. Equivalently,
\begin{equation}\label{QHDG}
    Q_{2N}(h) = (-1)^N D_{2N-1}^{res}\left(h;-\frac{n}{2}+N-1\right) \partial_r (\log v).
\end{equation}
\end{theorem}

\begin{proof} On the one hand, we have
$$
   D_{2N}^{res}\left(h;-\frac{n}{2}+N\right)(1) = P_{2N}(h)(1) = (-1)^N \left(\frac{n}{2}-N\right)
   Q_{2N}(h)
$$
using \eqref{Q-def} and \eqref{RF-Fact}. On the other hand, Corollary
\ref{RFvsSF} implies
$$
   D_{2N}^{res}\left(h;-\frac{n}{2}+N\right)(1) =
   \frac{1}{(-2N)_N(\frac{1}{2}-N)_N} \S_{2N}\left(h;\frac{n}{2}-N\right)(1).
$$
But
\begin{equation*}
   \S_{2N}\left(h;\frac{n}{2}-N\right) = \S_{2N-1}\left(h;\frac{n}{2}-N\right)
   S\left(g_+;\frac{n}{2}+N-1\right)
\end{equation*}
and
$$
   S\left(g_+;\frac{n}{2}+N-1\right)(1) = \left(\frac{n}{2}-N\right) \partial_r (\log v).
$$
By comparing both expressions for
$D_{2N}^{res}\left(h;-\frac{n}{2}+N\right)(1)$, we obtain
$$
   Q_{2N}(h) = (-1)^N \frac{1}{(-2N)_N(\frac{1}{2}-N)_N}
   \S_{2N-1}\left(h;\frac{n}{2}-N\right) \partial_r (\log v).
$$
Now simplification proves the first assertion. The second follows from this
result using the second relation in Corollary \ref{RFvsSF}.
\end{proof}

Theorem \ref{Q-holo-new} obviously follows from Theorem \ref{Q-holo-g} by
analytic continuation in the dimension $n$. However, the above proof avoids
this argument.

Finally, we prove that the formulas \eqref{Q-S} and \eqref{QHDG} are equivalent
to the well-known holographic formulas for $Q$-curvatures proved in \cite{GJ,
J-holo}. In fact, the definitions imply that \eqref{QHDG} is equivalent to
\begin{align*}
   (-1)^N Q_{2N} & = 2^{2N-2} (N-1)!^2 \\
   & \times \left( \sum_{j=0}^{2N-1} \T_j^*(\tfrac{n}{2}-N) v_0 + \cdots +
   \T_0^*(\tfrac{n}{2}-N) v_j\right) \frac{1}{(2N-1-j)!} \iota^* \partial_r^{2N-1-j}
   \left(\frac{\dot{v}}{v}\right).
\end{align*}
In the latter sum, the operator $\T_{2k}^*(\tfrac{n}{2}-N)$ acts on
$$
   \left(v_0 \frac{1}{(2N-1-2k)!} \iota^* \partial_r^{2N-1-2k} + \cdots + v_{2N-2k-2} \iota^* \partial_r\right)
   \left(\frac{\dot{v}}{v}\right).
$$
But this sum equals the $(2N-1-2k)^{th}$ Taylor coefficient of $v (\dot{v}/v) =
\dot{v}$, i.e., equals $(2N-2k) v_{2N-2k}$. Hence the above formula simplifies
to
$$
   (-1)^N Q_{2N} = 2^{2N-2} (N-1)!^2 \sum_{k=0}^{N-1} (2N-2k) \T_{2k}^*(\tfrac{n}{2}-N) (v_{2N-2k}).
$$
This formula is equivalent to \cite[Theorem 1.1]{J-holo}. The same arguments
also apply in the critical case. This completes the proof.

In summary, the above discussion provides reformulations and easy new proofs of
well-known holographic formulas for $Q$-curvatures (in even dimension). The key
arguments here are the second identity in Corollary \ref{RFvsSF} and/or the
first identity in Corollary \ref{Factor-residue}.

\section{A panorama of examples}\label{panorama}

In the present section, we illustrate the main results (Theorem \ref{DeltaN},
Theorem \ref{BigGJMS}, Theorem \ref{BSOperatorVsGJMS}, Proposition
\ref{GJMSVsBS} and Theorem \ref{Q-holo-new}) by low-order examples.
Furthermore, we discuss certain remarkable expansion of the families
$S_N(g_+;\lambda)$ with respect to the parameter $r$.

We shall often simplify notation by omitting the obvious metrics. In
particular, we shall write $\Delta$ for $\Delta_h$, $\J$ for $\J(h)$, $\Rho$
for $\Rho(h)$, $S(\lambda)$ for $S(g_+;\lambda)$ etc.

We also recall the expansion (\cite[Section 6.11]{J1})
\begin{equation}\label{eq:r-Expansions}
    \Delta_{\bar{g}} f = [\Delta f + \partial_r^2 f] -r \J \partial_r f + r^2 \left[-\delta (\Rho \#d f)
    -\frac{1}{2} (d\J,df)\right] + O(r^3).
\end{equation}
By $\dot{v}/v = 2r v_2 + O(r^3)$ with $v_2 = -\frac{1}{2} \J(h)$, we have
\begin{equation*}
    S(g_+;\lambda) = -(2\lambda\!-\!n\!+\!1) \partial_r + r[\Delta_h + \partial_r^2 +
   (\lambda\!-\!n\!+\!1) \J(h))] + O(r^2).
\end{equation*}

\subsection{Theorem \ref{DeltaN} for $N \le 3$}

The first-order family $\delta_1(\lambda)$ is given by (\cite[Section 6.2]{J1})
\begin{equation*}
   \delta_1(\lambda) = \iota^*\partial_r.
\end{equation*}
 On the other hand, we have
\begin{equation}\label{S1}
   \S_1(\lambda) = \iota^* S_1(\lambda) = -(2\lambda\!-\!n\!+\!1) \iota^*\partial_r
   = -(2\lambda\!-\!n\!+\!1) \delta_1(\lambda).
\end{equation}
This confirms Theorem \ref{DeltaN} in the case $N=1$.

The second-order family $\delta_2(\lambda)$ is given by \cite[Section 6.7]{J1}
\begin{equation*}
   \delta_2(\lambda) = \frac{1}{2} \iota^* \partial_r^2
   + \frac{1}{2(n\!-\!2\!-\!2\lambda)}(\Delta + (\lambda\!-\!n\!+\!2)\J)\iota^*.
\end{equation*}
On the other hand, by definition, we have
\begin{align*}
   \S_2(\lambda ) = \iota^* S_2(\lambda) & =\iota^* S(\lambda) S(\lambda+1)\\
   & = -(2\lambda\!-\!n\!+\!1)\iota^* \partial_r[r\Delta_{\bar{g}} + (-2\lambda\!+\!n\!-\!3)\partial_r
   + (\lambda\!-\!n\!+\!2)r \J].
\end{align*}
By \eqref{eq:r-Expansions}, this formula simplifies to
\begin{align}\label{S2}
   \S_2(\lambda) & = -(2\lambda\!-\!n\!+\!1)[\Delta \iota^* - (2\lambda\!-\!n\!+\!2) \iota^*\partial^2_r
   + (\lambda\!-\!n\!+\!2) \J \iota^*] \notag \\
   & = (-2)_2(2\lambda\!-\!n\!+\!1)_2 \left[\frac{1}{2} \iota^* \partial^2_r
   - \frac{1}{2(2\lambda\!-\!n\!+\!2)}(\Delta \iota^* + (\lambda\!-\!n\!+\!2) \J \iota^*) \right]\\
   & = (-2)_2 (2\lambda\!-\!n\!+\!1)_2 \delta_2(\lambda). \notag
\end{align}
This identity confirms Theorem \ref{DeltaN} for $N=2$.

Note that, for $n=3$, the latter formula gives
$$
   \S_2(\lambda) (1) = - (-2)_2 (2\lambda-2)_2 \frac{(\lambda-1)}{2(2\lambda-1)} \J = - 2(\lambda-1)^2 \J.
$$
This confirms the double zero mentioned after Proposition \ref{vanish}. In this
case, both factors in the definition of $\S_2(\lambda)$ contribute a zero.

The third-order family $\delta_3(\lambda)$ is given by (\cite[Section 6.8]{J1})
\begin{equation*}
   \delta_3(\lambda) =
   \frac{1}{6} \iota^*\partial_r^3 + \frac{1}{2(n\!-\!2\!-\!2\lambda)}
   (\Delta +(\lambda\!-\!n\!+\!2)\J) \iota^*\partial_r.
\end{equation*}
On the other hand, by definition, we have
\begin{align*}
   \S_3(\lambda) = \iota^* S_3(\lambda) & = \iota^* S(\lambda) S(\lambda+1) S(\lambda+2)\\
   & =-(2\lambda\!-\!n\!+\!1)\iota^*\partial_r [r\Delta_{\bar{g}}-(2\lambda\!-\!n\!+\!3)\partial_r
   +(\lambda\!-\!n\!+\!2) r\J] \\
   & \quad \circ [r\Delta_{\bar{g}} - (\lambda\!-\!n\!+\!5)\partial_r + (\lambda\!-\!n\!+\!3)r\J].
\end{align*}
By \eqref{eq:r-Expansions}, the last display simplifies to
\begin{align}\label{S3}
   & -(2\lambda\!-\!n\!+\!1) \bigg[(-(\lambda\!-\!n\!+\!5) + 2 - 2(2\lambda\!-\!n\!+\!3)) \Delta \iota^* \partial_r \notag \\
   & +(2\!-\!(2\lambda\!-\!n\!+\!5) - 2(2\lambda\!-\!n\!+\!3) + (2\lambda\!-\!n\!+\!3)(2\lambda\!-\!n\!+\!5)) \iota^*\partial_r^3 \notag \\
   & +(-2\!+\!2(\lambda\!-\!n\!+\!3) + 2(2\lambda\!-\!n\!+\!3) - 2(2\lambda\!-\!n\!+\!3)(\lambda\!-\!n\!+\!3) \notag \\
   & -(\lambda\!-\!n\!+\!2)(2\lambda\!-\!n\!+\!5)) \J \iota^*\partial_r \bigg] \notag \\
   & =(-3)_3(2\lambda\!-\!n\!+\!1)_3 \left[\frac{1}{6} \iota^*\partial_r^3 - \frac{1}{2(2\lambda\!-\!n\!+\!2)}
   \big[\Delta \iota^*\partial_r + (\lambda\!-\!n\!+\!2) \J \iota^*\partial_r\big]\right] \notag \\
   & =(-3)_3 (2\lambda\!-\!n\!+\!1)_3 \delta_3(\lambda).
\end{align}
This identity confirms Theorem \ref{DeltaN} for $N=3$.

\subsection{Theorem \ref{BigGJMS} for $N = 1$}

We have shown that the family $S_N(g_+;\lambda)$ contains the GJMS operator
$P_{2N}(\bar{g})$ if $g_+$ is an Einstein metric on $M_+^\circ$. Now we give a
direct proof for the first example. We recall that $m=\frac{n+1}{2}$. By
formula \eqref{Jbar}, we compute
\begin{equation*}
   S_1(g_+;\m-1) = S(g_+;\m-1) = r \left(\Delta_{\bar{g}}-\frac{n-1}{2}\J(\bar{g})\right)
\end{equation*}
if $g_+$ is Einstein. By \eqref{Yamabe}, the right-hand side coincides with $r
P_2(\bar{g})$. For general Poincar\'e metrics, the identity holds with error
terms (see Remark \ref{AE}).

\subsection{Theorem \ref{BSOperatorVsGJMS} for $N=1$}

It shows that the family $S_{2N}(g_+;\lambda)$ induces a tangential operator
which is proportional to the GJMS operators for $(M,h)$. Now we compute
\begin{align*}
   \iota^* S_2\left(\frac{n}{2}-1\right) & = \iota^* S \left(\frac{n}{2}-1\right) S\left(\frac{n}{2}\right)\\
   & = \iota^*\partial_r \left[-\partial_r + r \left[\Delta + \partial_r^2 - \left(\frac n2-1\right) \J\right]\right]\\
   & = \Delta\iota^*-\left(\frac{n-2}{2}\right) \J \iota^*,
\end{align*}
which coincide with $P_2 \iota^*$ (see \eqref{Yamabe}).

Note also that
\begin{equation*}
   \iota^*S_2 \left(g_+;\frac{n\!-\!1}{2}\!-\!1\right) = 2
   \iota^* \partial_r \left[r \Delta_{\bar{g}} - \left(\frac{n-1}{2}\right) r\J \right]
   = 2 \iota^* P_2(\bar{g})
\end{equation*}
and
\begin{equation*}
   \iota^* S_3\left(g_+;\frac{n\!-\!3}{2}\!-\!1\right) = 4! \iota^* \partial_r P_2(\bar{g})
\end{equation*}
by \eqref{S2} and \eqref{S3} (for Einstein $g_+$ see also
\eqref{eq:RExpansion2} and \eqref{S3-rep}). The latter two identities are
special cases of \eqref{Seven-special} and \eqref{Sodd-special}.

\subsection{Proposition \ref{GJMSVsBS} for $N \le 2$}

We recall that (\cite[Section $6.7$]{J1})
\begin{equation}\label{eq:SolutionOperator2}
    \T_2(\lambda) = \frac{1}{2(n\!-\!2\lambda\!-\!2)}[\Delta - \lambda \J].
\end{equation}
By \eqref{eq:ExpansionOfP}, we have
\begin{equation*}
    S^{(0)}(n\!-\!\lambda\!-\!1) = \Delta + \partial_r^2 - \lambda \J.
\end{equation*}
Hence we obtain
\begin{equation*}
    -2(2\lambda\!-\!n\!+\!2) \T_2(\lambda) = S^{(0)}(n\!-\!\lambda\!-\!1) \T_0(\lambda)
\end{equation*}
as an identity of operators on $C^\infty(M)$. This coincides with \eqref{algo}
for $N=1$.

We recall that (\cite[Theorem $6.9.4$]{J1})
\begin{equation*}
   \T_4(\lambda) = \frac{1}{4(n\!-\!4\!-\!2\lambda)}
   \bigg[\frac{1}{2(n\!-\!2\!-\!2\lambda)}[\Delta - (\lambda\!+\!2) \J][\Delta\!-\!\lambda \J]
   -\frac 12 \lambda \abs{\Rho}^2 - \delta (\Rho \#d) - \frac{1}{2}(d \J,d) \bigg].
\end{equation*}
Now, using \eqref{eq:PCoefficients} and \eqref{eq:SolutionOperator2}, we
compute
\begin{multline*}
   S^{(2)}(n\!-\!\lambda\!-\!1) \T_0(\lambda) + S^{(0)}(n\!-\!\lambda\!-\!3) \T_2(\lambda)\\
   = -\delta (\Rho \# d) - \frac 12 (d\J,d) - \frac{1}{2} \abs{\Rho}^2 + \frac{1}{2(n\!-\!2\lambda\!-\!2)}
   [\Delta - (\lambda\!+\!2) \J][\Delta - \lambda \J].
\end{multline*}
Hence
\begin{equation*}
   -4(2\lambda\!-\!n\!+\!4) \T_4(\lambda) = S^{(2)}(n\!-\!\lambda\!-\!1) \T_0(\lambda)
   + S^{(0)}(n\!-\!\lambda\!-\!3)\T_2(\lambda).
\end{equation*}
This coincides with \eqref{algo} for $N=2$.

\subsection{Theorem \ref{Q-holo-new} for $n \le 4$ and Theorem \ref{Q-holo-g} for $N \le 2$}\label{Q-ex}

We first confirm \eqref{Q-S} for $n=2$ and $n=4$. As a preparation, we note
that
$$
  \partial_r (\log v) = \dot{v}(r)/v(r) = 2 r v_2 + r^3 (4 v_4 - 2 v_2^2) + \cdots.
$$
For $n=2$, \eqref{Q-S} claims that
$$
   Q_2 = - \S_1(0) \partial_r (\log v).
$$
Now, using \eqref{S1}, this identity simplifies to
$$
   Q_2 = -2 v_2.
$$
For $n=4$, \eqref{Q-S} claims that
$$
   Q_4 = \frac{4}{36} \S_3(0) \partial_r (\log v).
$$
Using \eqref{S3}, this formula reads
$$
   Q_4 = 4 \left[\frac{1}{6} \iota^*\partial_r^3 + \frac{1}{4}
   (\Delta \iota^* \partial_r - 2\J \iota^* \partial_r )\right] \partial_r (\log v).
$$
The latter expression simplifies to
$$
   16 v_4 - 8v_2^2 + 2(\Delta - 2\J) v_2.
$$
Now the standard formulas
$$
   v_2 = -\frac{1}{2} \J \quad \mbox{and} \quad v_4 = \frac{1}{8} (\J^2 - |\Rho|^2)
$$
confirm that the above formulas are equivalent to the well-known expressions
$$
   Q_2 = \J \quad \mbox{and} \quad Q_4 = 2\J^2 - 2 |\Rho|^2 - \Delta \J.
$$
Similar calculations confirm \eqref{QHG} for $N=1$ and $N=2$. In fact, by
\eqref{S1}, the assertion
$$
   Q_2 = - \S_1\left(\frac{n}{2}-1\right) \partial_r (\log v)
$$
is equivalent $Q_2 = -2 v_2$. Similarly, by \eqref{S3}, the assertion
$$
   Q_4 = \frac{4}{36} \S_3\left(\frac{n}{2}-2\right) \partial_r(\log v)
$$
is equivalent to
$$
   Q_4 = 16 v_4 - 8 v_2^2 + 2 \Delta v_2 - n \J v_2 = \frac{n}{2} \J^2 - 2|\Rho|^2 - \Delta
   \J.
$$
These results reproduce the formulas in \eqref{sub-Q}.

\subsection{Some interesting expansion of $S_{N}(g_+;\lambda)$}\label{expansions}

We finish this section with a discussion of some interesting expansions of the
families $S_{N}(g_+;\lambda)$. Here we restrict to the low-order examples $N
\le 3$. In order to simplify the presentation, we shall also assume that $g_+$
is Einstein. The expansions in question describe the families
$S_N(g_+;\lambda)$ as polynomials in $r$ with coefficients that are polynomials
in the GJMS operators $P_{2k}(\bar{g})$ for $k\leq N$ and the operators
$\partial_r^w = w^{-1}\partial_r(w\cdot)$. The existence of such expansions
follows from Theorem \ref{BigGJMS}, Corollary \ref{IdentitiesForBSOperator} and
Proposition \ref{LeadingTermForBSOperator}. Of particular interest will be the
coefficient of $r^N$ in the expansion of $S_{N}(h;\lambda)$. This coefficient
is a polynomial in $\lambda$ the leading coefficient of which is related to the
second-order operators $\M_2(\bar{g})$, $\M_4(\bar{g})$ and $\M_6(\bar{g})$
(see \eqref{M4-bar}, \eqref{M6-bar}) associated to $(M_+,\bar{g})$. We recall
that these operators are the first coefficients in the $r$-expansion of the
holographic Laplacian introduced in \cite{J2}. These experiments provide
evidence for a general property of these expansions which will be formulated at
the end.

As in Theorem \ref{BigGJMS}, we shall use the notation $m=\frac{n+1}{2}$.

\begin{example} The second-order family $S_1(g_+;\lambda)$ is linear in the variable
$\lambda$ and satisfies two identities:
\begin{align*}
   S_1(g_+;m\!-\!1) & = r P_2(\bar{g}),\\
   \frac{d}{d \lambda} S_{1}(g_+;\lambda) & = -2 \partial_r^w.
\end{align*}
These identities imply the representation
\begin{equation}\label{eq:LambdaExpansion1}
    S_1(g_+;\lambda) = -2(\lambda\!-\!\m\!+\!1)\partial_r^w + r P_2(\bar{g}).
\end{equation}
In particular, the coefficient of $r$ is given by $\M_2(\bar{g}) =
P_2(\bar{g})$.
\end{example}

\begin{example} The fourth-order family $S_{2}(g_+;\lambda)$ is quadratic in the
variable $\lambda$ and satisfies three identities:
\begin{align*}
   S_2(g_+;\m\!-\!1) & = r^2 P_4(\bar{g}),\\
   S_2(g_+;\m\!-\!2) & = S_1(g_+;\m\!-\!2) r P_2(\bar{g}),\\
   \frac{1}{2!} \frac{d^2}{d \lambda^2}S_{2}(g_+;\lambda) & = 4(\partial_r^w)^2.
\end{align*}
Hence $S_2(g_+;\lambda)$ can be written in the form
\begin{multline}\label{eq:Ansatz2}
   S_{2}(g_+;\lambda) = 4 (\lambda\!-\!\m\!+\!1)_2 (\partial_r^w)^2 \\
   -(\lambda\!-\!\m\!+\!1) S_1(g_+;\m\!-\!2)r P_2(\bar{g})
   +(\lambda\!-\!\m\!+\!2)r^2 P_4(\bar{g}).
\end{multline}
Now applying Lemma \ref{CommShiftM} to the middle summand gives
\begin{equation*}
   S_1(g_+;\m\!-\!2)r P_2(\bar{g}) = r S_1(g_+;\m\!-\!3) P_2(\bar{g}) + 2P_2(\bar{g}).
\end{equation*}
Thus, by combination with \eqref{eq:LambdaExpansion1}, we can rewrite
\eqref{eq:Ansatz2} as
\begin{align}\label{eq:RExpansion2}
   S_{2}(g_+;\lambda) & = 4 (\lambda\!-\!\m\!+\!1)_2(\partial_r^w)^2 -
   2(\lambda\!-\!\m\!+\!1)P_2(\bar{g})\notag\\
   &-4(\lambda\!-\!\m\!+\!1) r \partial_r^w P_2(\bar{g})\notag\\
   &+r^2[(\lambda\!-\!\m\!+\!2)P_4(\bar{g}) - (\lambda\!-\!\m\!+\!1)P_2(\bar{g})^2].
\end{align}
In particular, the coefficient of $r^2$ is a linear polynomial in $\lambda$ the
leading coefficient of which is given by the operator
\begin{equation}\label{M4-bar}
   \M_4(\bar{g}) = P_4(\bar{g}) - P_2(\bar{g})^2.
\end{equation}
\end{example}

\begin{example} The sixth-order family $S_{3}(g_+;\lambda)$ is cubic in the variable
$\lambda$ and satisfies four identities:
\begin{align*}
   S_{3}(g_+;\m\!-\!1) & = r^3 P_6(\bar{g}),\\
   S_{3}(g_+;\m\!-\!2) & = S_1(g_+;m\!-\!2) r^2 P_4(\bar{g}),\\
   S_{3}(g_+;\m\!-\!3) & = S_2(g_+;\m\!-\!3) r P_2(\bar{g}),\\
   \frac{1}{3!}\frac{d^3}{d \lambda^3}S_{3}(g_+;\lambda) & = -8(\partial_r^w)^3.
\end{align*}
Hence $S_{3}(g_+;\lambda)$ can be represented in the form
\begin{align}\label{eq:Ansatz3}
   S_{3}(g_+;\lambda) & = -8(\lambda\!-\!\m\!+\!1)_3 (\partial_r^w)^3\notag\\
   & + \frac{1}{2} (\lambda\!-\!\m\!+\!1)_2 S_2(g_+;\m\!-\!3)r P_2(\bar{g})\notag\\
   & -(\lambda\!-\!\m\!+\!1)(\lambda\!-\!\m\!+\!3) S_1(g_+;\m\!-\!2) r^2 P_4(\bar{g})\notag\\
   & + \frac{1}{2} (\lambda\!-\!\m\!+\!2)_2 r^3 P_6(\bar{g}).
\end{align}
Applying Lemma \ref{CommShiftM} to the middle two summands yields the relations
\begin{align*}
   S_2(g_+;\m\!-\!3) (r P_2(\bar{g}) )& = r S_2(g_+;\m\!-\!4) P_2(\bar{g}) + 6 S_1(g_+;\m\!-\!3) P_2(\bar{g}),\\
   S_1(g_+;\m\!-\!2) (r^2 P_4(\bar{g})) & = r^2 S_1(g_+;\m\!-\!4)P_4(\bar{g}) + 6r P_4(\bar{g}).
\end{align*}
Hence, by combination with the respective representations
\eqref{eq:LambdaExpansion1} and \eqref{eq:RExpansion2} of $S_1(g_+;\lambda)$
and $S_2(g_+;\lambda)$, we can rewrite $S_3(g_+;\lambda)$ in the form
\begin{align}\label{S3-rep}
   & (2\lambda\!-\!n\!+\!1)(2\lambda\!-\!n\!+\!3) (3\partial_r^w P_2(\bar{g})
   -(2\lambda\!-\!n\!+\!5)(\partial_r^w)^3) \notag\\
   & -\frac{3}{2} (2\lambda\!-\!n\!+\!1)r \left[(2\lambda\!-\!n\!+\!5)P_4(\bar{g})
   -(2\lambda\!-\!n\!+\!3)P_2(\bar{g})^2
   -2(2\lambda\!-\!n+\!3)(\partial_r^w)^2P_2(\bar{g}) \right]\notag\\
   & -\frac{3}{2}(2\lambda\!-\!n\!+\!1) r^2 \left[(2\lambda\!-\!n\!+\!5)\partial_r^wP_4(\bar{g})-
   (2\lambda\!-\!n\!+\!3) \partial_r^w P_2(\bar{g})^2 \right]\notag\\
   & +\frac{1}{8} r^3 \bigg[ -2(2\lambda\!-\!n\!+\!1)(2\lambda\!-\!n\!+\!5)P_2(\bar{g})P_4(\bar{g})+
   (2\lambda\!-\!n\!+\!3)(2\lambda\!-\!n\!+\!5)P_6(\bar{g})\notag\\
   & \qquad \; \quad -2(2\lambda\!-\!n\!+\!1)(2\lambda\!-\!n\!+\!3)P_4(\bar{g})P_2(\bar{g})
   +3(2\lambda\!-\!n\!+\!1)(2\lambda\!-\!n\!+\!3)P_2(\bar{g})^3 \bigg].
\end{align}
In particular, the coefficient of $r^3$ is a quadratic polynomial in $\lambda$
the leading coefficient of which is a constant multiple of
\begin{equation}\label{M6-bar}
   \M_6(\bar{g}) = P_6(\bar{g})-2 P_2(\bar{g}) P_4(\bar{g})-2P_4(\bar{g})P_2(\bar{g})+ 3P_2(\bar{g})^3.
\end{equation}
\end{example}

The above examples suggest that in the analogous representation of
$S_N(g_+;\lambda)$ the coefficient of $r^N$ is a constant multiple of the
degree $N-1$ polynomial
\begin{equation}\label{leading-coeff}
    \sum_{|I|=N} m_I \frac{(\lambda-\frac{n-1}{2})_N}{\lambda-\frac{n-1}{2}+N-I_r} P_{2I}(\bar{g}).
\end{equation}
Here the sum runs over all partitions $I = (I_1,\dots,I_r)$ of size $|I|=N$
and, for any $I$, we set $P_{2I} = P_{2I_1} \cdots P_{2I_r}$; for more details
on the coefficients $m_I$ we refer to \cite{J2}. The expression
\eqref{leading-coeff} resembles the polynomials in \cite[Theorem 4.1]{J2} which
describes residue families in terms of GJMS operators. Note that the leading
coefficient of the polynomial \eqref{leading-coeff} is the building block
operator
$$
   \M_{2N}(\bar{g}) = \sum_{|I|=N} m_I P_{2I}(\bar{g})
$$
(see Section \ref{epi}). We will return to that problem in later work.

\section{Epilogue}\label{epi}

In the present section, we sketch some further developments and indicate some
interesting future developments.

We first prove a consequence of the expansions discussed in Section
\ref{expansions}. Then we generalize this result to arbitrary order. We show
that the result naturally follows from a basic property of an operator which in
\cite{J3} was termed the holographic Laplacian. We expect that these
generalizations play a similar role in the study of the families
$S_N(g_+;\lambda)$ of arbitrary order $N \in \N$.

In order to simplify the presentations as much as possible, we assume that $M$
is an analytic manifold of odd dimension $n$. For an analytic metric $h$ on
$M$, we let $g_+$ be a Poincar\'e metric in normal form relative to $h$. It
satisfies $\Ric(g_+) + n g_+ = 0$ and, for any $N \in \N$, there is a
well-defined GJMS operator $P_{2N}(h)$.

We also recall the notation $\partial_r^w = w^{-1}\partial_r(w\cdot)$ and
$$
   \M_2(h) = P_2(h), \quad \M_4(h) = P_4(h) - P_2(h)^2.
$$
As usual, let $\bar{g} = r^2 g_+$. The analogous definitions yield
$\M_2(\bar{g})$ and $\M_4(\bar{g})$.

\begin{theorem}\label{magic} $r \M_4(\bar{g}) = 2[\partial_r^w,\M_2(\bar{g})]$.
\end{theorem}

\begin{proof} Using \eqref{eq:LambdaExpansion1}, we compute
\begin{align*}
   S_2(g_+;\lambda) & = S_1(g_+;\lambda) S_1(g_+;\lambda+1) \\
   & = S_1(g_+;\lambda) \left[-2(\lambda\!-\!\m\!+\!2)\partial_r^w + r P_2(\bar{g})\right] \\
   & = -2(\lambda\!-\!\m\!+\!2) S_1(g_+;\lambda) \partial_r^w + S_1(g_+;\lambda) r P_2(\bar{g}).
\end{align*}
Now we apply Lemma \ref{GeneralSL(2)} to move the variable $r$ in the last term
to the left, i.e.,
\begin{equation*}
   S_2(g_+;\lambda) = -2(\lambda\!-\!\m\!+\!2) S_1(g_+;\lambda)\partial_r^w + r S_1(g_+;\lambda)
   P_2(\bar{g}) - 2(\lambda\!-\!\m\!+\!1) P_2(\bar{g})).
\end{equation*}
By another application of \eqref{eq:LambdaExpansion1}, we conclude
\begin{align*}
   S_2(g_+;\lambda) & = 2(\lambda\!-\!\m\!+\!2)[-2(\lambda\!-\!\m\!+\!1)(\partial_r^w)^2 + rP_2(\bar{g})\partial_r^w]\\
   & + r[-2(\lambda\!-\!\m) \partial_r^w P_2(\bar{g}) + P_2(\bar{g})^2 - 2(\lambda\!-\!\m\!+\!1)P_2(\bar{g})\\
   & = 4(\lambda\!-\!\m\!+\!1)_2 (\partial_r^w)^2 - 2(\lambda\!-\!\m\!+\!1) P_2(\bar{g})\\
   & -2r[(\lambda\!-\!\m\!+\!2) P_2(\bar{g}) \partial_r^w + (\lambda\!-\!\m) \partial_r^wP_2(\bar{g})] + r^2P_2(\bar{g})^2.
\end{align*}
The difference between \eqref{eq:RExpansion2} and the last display gives the
relation
\begin{equation*}
   0 = (\lambda\!-\!\m\!+\!2) \left[r^2 \M_4(\bar{g}) - 2r[\partial_r^w,P_2(\bar{g})]\right].
\end{equation*}
The proof is complete.
\end{proof}

Now we describe an alternative proof of the commutator relation in Theorem
\ref{magic}. The proof rests on a basic property of the {\em holographic
Laplacian} $\H(h)(r)$ introduced in \cite{J2}, \cite{J3}. We recall that this
operator is the Schr\"{o}dinger-type operator
\begin{equation}\label{sch}
   \H(h)(r) \st - \delta_h (h_r^{-1} d) + \U(h)(r)
\end{equation}
with the potential
\begin{equation}\label{potential}
   \U(h)(r) \st - w(r)^{-1} \left(\partial^2/\partial r^2 - (n-1) r^{-1}
   \partial/\partial r - \delta (h_r^{-1} d) \right) (w(r)).
\end{equation}
The operator $\H(h)(r)$ should be viewed as a $1$-parameter deformation of the
Yamabe operator $\H(h)(0) = P_2(h)$. It is a key fact (see \cite{FG-J},
\cite{J2}) that (the Taylor series of) this operator coincides with
$$
   \G(h)\left(\frac{r^2}{4}\right),
$$
where
\begin{equation*}
    \G(h)(\rho) \st \sum_{N \ge 1} \M_{2N}(h) \frac{\rho^{N-1}}{(N-1)!^2}
\end{equation*}
is a generating function of the so-called building block operators $\M_{2N}(h)$
of the GJMS operators of the metric $h$. In fact, any GJMS operator $P_{2N}(h)$
can be written as a linear combination
\begin{equation}\label{BB}
   P_{2N}(h) = \sum_{|I|=N} n_I \M_{2I}(h), \; n_I \in \Z
\end{equation}
of compositions $\M_{2I} = \M_{2I_1} \cdots \M_{2I_r}$ for $I=(I_1,\dots,I_r)$.
For the details we refer to \cite{J2}. Now we consider the generating function
$\G(\bar{g}(r))(\eta)$. It satisfies the basic relation
\begin{equation}\label{fund}
    \G(\bar{g}(r))(\eta) = w(r)^{-1} \G(h)\left(\frac{r^2}{4}+\eta\right) w(r) + (\partial_r^w)^2
\end{equation}
of second-order operators acting on functions in $(r,x)$. Note that the
operator in the first term on the right-hand side only differentiates along
$M$. The identity \eqref{fund} follows by combining the relation between
$\G(\bar{g})$ and the holographic Laplacian of $\bar{g}$ with an explicit
formula \cite[Theorem 7.2]{J2} for the Poincar\'e metric in normal form
relative to $\bar{g}$ in terms of the Poinacar\'e metric in normal form
relative to $h$ ; for the details we refer to \cite{J-bb}. By expansion into
powers of $\eta$, the relation \eqref{fund} implies the identities
\begin{equation}\label{Mbexp}
    \M_{2N}(\bar{g}(r)) =
    \sum_{k\geq N} w(r)^{-1} \M_{2k}(h) w(r) \frac{(N-1)!}{(k-1)!(k-N)!}\left(\frac{r^2}{4}\right)^{k-N}
\end{equation}
for $N \ge 2$. In turn, the latter relation yields the following commutator
relations.

\begin{theorem}\label{magic-2} Let $N\geq 2$. Then
\begin{equation}\label{eq:help1}
      r \M_{2N}(\bar{g}) = 2(N-1)[\partial_r^w,\M_{2N-2}(\bar{g})].
\end{equation}
\end{theorem}

\begin{proof} Let $N \ge 3$. We use \eqref{Mbexp} and its relative
\begin{equation}\label{MBexp-2}
   \M_{2N-2}(\bar{g}(r)) = \sum_{k \ge N-1} w^{-1} \M_{2k}(h) w \frac{(N-2)!}{(k-1)!(k-N+1)!}
   \left(\frac{r^2}{4}\right)^{k-N+1}
\end{equation}
to calculate
\begin{align*}
   [\partial_r^w,\M_{2N-2}(\bar{g})] & = \partial_r^w \M_{2N-2}(\bar{g}) -
   \M_{2N-2}(\bar{g}) \partial_r^w \\
   & = w^{-1} \partial_r w \M_{2N-2}(\bar{g}) - \M_{2N-2}(\bar{g}) w^{-1} \partial_r w \\
   & = w^{-1} \partial_r \left(\sum_{k \ge N-1} \M_{2k}(h) w \frac{(N-2)!}{(k-1)!(k-N+1)!}
   \left(\frac{r^2}{4}\right)^{k-1} \cdot \right) \\
   & \qquad \qquad  -
   \sum_{k \ge N-1} w^{-1} \M_{2k}(h) \frac{(N-2)!}{(k-1)!(k-N+1)!}
   \left(\frac{r^2}{4}\right)^{k-1} \partial_r (w \cdot) \\
   & = \frac{r}{2} \sum_{k \ge N-2} w^{-1} \M_{2k}(h) w \frac{(N-2)!}{(k-1)! (k-N)!}
   \left(\frac{r^2}{4}\right)^{k-2} \partial_r \\
   & =  \frac{r}{2(N-1)} \M_{2N}(\bar{g}).
\end{align*}
This proves the assertion for $N \ge 3$. For $N=2$, the identity
\eqref{MBexp-2} is to be replaced by
$$
    \M_{2}(\bar{g}) = \sum_{k \ge 1} w^{-1} \M_{2k}(h) w \frac{1}{(k-1)!(k-1)!}
   \left(\frac{r^2}{4}\right)^{k-1} + (\partial_r^w)^2.
$$
Since the additional term $(\partial_r^w)^2$ commutes with $\partial_r^w$, the
above arguments extend to that case.
\end{proof}

Theorem \ref{magic} is the special case $N=2$ of Theorem \ref{magic-2}.

\begin{example} Let $M$ be the sphere $S^n$ with the round metric $g_{S^n}$.
Then
$$
  \bar{g}(r) = dr^2 + (1-r^2/4)^2 g_{S^n}.
$$
But $\M_2(\bar{g}) = P_2(\bar{g})$ and
$$
   \M_{2N}(\bar{g}) = (N-1)! N! (1-r^2/4)^{-N-1} P_2(g_{S^n})
$$
for $N \ge 2$. These results can be derived from the identification of the
holographic Laplacian $\H(\bar{g})$ with the generating series $\G(\bar{g})$ of
the operators $\M_{2N}(\bar{g}$). For a direct proof see \cite[Section
11.10]{J2}. Moreover, we have
$$
   \partial^w_r = \partial_r - \frac{n}{4} r (1-r^2/4)^{-1}.
$$
Hence we calculate
\begin{align*}
   2(N-1) [\partial_r^w,\M_{2N-2}(\bar{g})] & = 2(N-1)! (N-1)! \partial_r
   ((1-r^2/4)^{-N}) P_2(g_{S^n}) \\
   & = N!(N-1)! (1-r^2/4)^{-N-1} r P_2(g_{S^n}) \\
   & = r \M_{2N}(\bar{g})
\end{align*}
for $N \ge 3$. This proves \eqref{eq:help1} for $N \ge 3$. A direct calculation
also confirms the case $N=2$.
\end{example}

In turn, Theorem \ref{magic-2} leads to a simple compressed formula for the
generating series $\G(\bar{g})$ and thus for the holographic Laplacian of
$\bar{g}$. In order to formulate the result, we introduce the following
notation. Let
\begin{equation*}
    R\circ \ad(\partial_r^w)(\cdot) \st \frac{1}{r} \circ [\partial_r^w,\cdot].
\end{equation*}

\begin{theorem}\label{holo-compressed} Assume that $(M^n,h)$ is real analytic of
odd dimension $n$. Then
\begin{equation}\label{eq:NewFormulaForHL}
   \H(\bar{g})(\eta) = \G(\bar{g})\left(\frac{\eta^2}{4}\right)
   = \exp \left( R\circ \ad(\partial_r^w) \frac{\eta^2}{2} \right)(\M_{2}(\bar{g})).
\end{equation}
\end{theorem}

\begin{proof} By a repeated application of the identity \eqref{eq:help1} we obtain
\begin{equation*}
   \M_{2N}(\bar{g})=2^{N-1}(N-1)! (R\circ \ad(\partial_r^w))^{N-1}(\M_{2}(\bar{g}))
\end{equation*}
for $N \ge 2$. Using the natural convention $(R\circ \ad(\partial_r^w))^0 =
\id$, the latter identity also holds for $N=1$. Hence
\begin{align*}
   \H(\bar{g})(\eta) & = \sum_{N\geq 1}\M_{2N}(\bar{g})
   \frac{1}{(N-1)!^2}\left(\frac{\eta^2}{4}\right)^{N-1}\\
   & = \sum_{N\geq 1} (R\circ \ad(\partial_r^w))^{N-1}(\M_{2}(\bar{g}))
   \frac{1}{(N-1)!}\left(\frac{\eta^2}{2}\right)^{N-1}\\
   & = \sum_{N\geq 0} (R \circ \ad(\partial_r^w))^{N}(\M_{2}(\bar{g}))
   \frac{1}{N!} \left(\frac{\eta^2}{2}\right)^{N}\\
   & = \exp \left(R\circ \ad(\partial_r^w)\frac{\eta^2}{2}\right)(\M_{2}(\bar{g})).
\end{align*}
This completes the proof.
\end{proof}

Theorem \ref{holo-compressed} again clearly shows that $\H(\bar{g})(\eta)$ is a
deformation of $\M_2(\bar{g}) = P_2(\bar{g})$.

We finish with brief comments on generalizations of the present theory to
differential forms. The theory of differential symmetry breaking operators on
functions has a natural extension to differential forms \cite{FJS,KKP}. Curved
versions of that theory deal with residue families acting on differential forms.
Their theory will be developed elsewhere. Residue families on differential forms are
expected to have analogous descriptions in terms of compositions of shift operators
on differential forms. Results in \cite{FOS} confirm that picture in the flat case.
Curved analogs of the degenerate Laplacian acting on differential forms were studied
in \cite{GLW} in terms of tractor calculus.



\begin{thebibliography}{GJMS92}

\bibitem [BJ]{BJ}
H.~Baum and A.~Juhl, {\em Conformal Differential Geometry: $Q$-curvature and
Conformal Holonomy}. Oberwolfach Seminars {\bf 40}, 2010.

\bibitem [B]{sharp}
T.~P.~Branson, Sharp inequalities, the functional determinant, and the
complementary series, {\em Trans. AMS} {\bf 347} (1995), 3671--3742.

\bibitem [C]{C}
J-L.~Clerc, \emph{Another approach to Juhl's conformally covariant differential
operators from $S^n$ to $S^{n-1}$}, SIGMA {\bf 13}, 11, (2017).
\url{arXiv:1612.01856}

\bibitem [FG1]{FG}
C.~Fefferman and C.R.~Graham, {\em The Ambient Metric}, Annals of Math. Studies
\textbf{178}, 2011. \url{arXiv:0710.0919}

\bibitem [FG2]{FG-J}
C.~Fefferman and C.~R.~Graham, {\em Juhl's formulae for GJMS operators and
$Q$-curvatures}. J. Amer. Math. Soc. 26, 4, (2013), 1191--1207.
\url{arXiv:1203.0360}

\bibitem [FH]{FH}
C.~Fefferman and K.~Hirachi, {\em Ambient metric construction of $Q$-curvature
in conformal and CR geometries}, Math. Res. Lett. {\bf 10}, 5-6, (2003),
819--832. \url{arXiv:math/0303184}

\bibitem [FJS]{FJS}
M.~Fischmann, A.~Juhl and P.~Somberg, {\em Conformal symmetry breaking
differential operators on differential forms}, Memoirs of AMS (2018) (to
appear). \url{arXiv:1605.04517}

\bibitem [F{\O}S]{FOS}
M.~Fischmann, B.~{\O}rsted and P.~Somberg, {\em Bernstein-Sato identities and
conformal symmetry breaking operators}, Journal of Functional Analysis {\bf 277},
11. \url{arXiv:1711.01546}

\bibitem [G]{Gover-AE}
R.~Gover, {\em Almost Einstein and Poincar\'e-Einstein manifolds in Riemannian
signature}, Joural of Geometry and Physics {\bf 60}, 2 (2010), 182--204.
\url{arXiv:0803.3510v1}

\bibitem [GH]{GH}
A.~R.~Gover and K.~Hirachi, {\em Conformally invariant powers of the {L}aplacian--a
complete nonexistence theorem}, J. Amer. Math. Soc. {\bf 17}, 2, (2004), 389--405.
\url{arXiv:math/0304082v2}

\bibitem [GLW]{GLW}
A.~R.~Gover, E.~Latini and A.~Waldron, {\em Poincar\'e-Einstein holography for forms
via conformal geometry in the bulk}, Mem. Amer. Math. Soc. {\bf 235} (2015) no.
1106. \url{arXiv:1205.3489}

\bibitem [GP]{GP}
R.~Gover, L.~Peterson, {\em Conformal boundary operators, $T$-curvatures, and
conformal fractional Laplacians of odd order}. \url{arXiv:1802.08366}

\bibitem [GW]{GW}
R.~Gover, A.~Waldron, {\em Boundary calculus for conformally compact manifolds},
Indiana Univ. Math. J. {\bf 63}, 1, (2014), 119--163. \url{arXiv:1104.2991v2}

\bibitem [G1]{non-ex}
C.~R.~Graham, {\em Conformally invariant powers of the Laplacian. II.
Nonexistence}, J. London Math. Soc. {\bf 46}, 2, (1992), 566--576.

\bibitem [G2]{G-vol}
C.~R.~Graham, Volume and area renormalizations for conformally compact
{E}instein metrics, {\em Rend. Circ. Mat. Palermo (2) Suppl.} {\bf 63} (2000),
31--42. {\url{arXiv:math/9909042v1}}

\bibitem [GJMS]{GJMS}
C.~R.~Graham, R.~Jenne, L.~J.~Mason and G.~A.~J.~Sparling, {\em Conformally
invariant powers of the {L}aplacian. {I}. {E}xistence}, J. London Math. Soc.
(2) {\bf 46}, 3, (1992), 557--565.

\bibitem [GJ]{GJ}
C.~R.~Graham and A. Juhl, {\em Holographic formula for $Q$-curvature}, Adv.
Math. {\bf 216}, (2007), 2, 841--853. {\url{arXiv:0704.1673v1}}

\bibitem [GL]{GL}
C.~R.~Graham and J.~M.~Lee, {Einstein metrics with prescribed conformal
infinity on the ball}, Adv. Math. {\bf 87} (1991), 186--225.

\bibitem [GZ]{GZ}
C.~R.~Graham and M.~Zworski, {\em Scattering matrix in conformal geometry},
Inventiones math. \textbf{152}, 1, (2003), 89--118. \url{arXiv:math/0109089}

\bibitem [He]{Hel}
S.~Helgason, {\em Groups and Geometric Analysis. Integral Geometry, Invariant
Differential Operators, and Spherical Functions}, Academic Press, 1984.

\bibitem [J1]{J1}
A.~Juhl, {\em Families of conformally covariant differential operators,
$Q$-curvature and Holography}, Birkh\"{a}user, Progress in Mathematics {\bf 275},
2009.

\bibitem [J2]{J-holo}
A.~Juhl, {\em Holographic formula for {$Q$}-curvature. II}, Adv. Math. {\bf
226} (2011), 3409--3425. \url{arXiv:1003.3989}

\bibitem [J3]{J4}
A.~Juhl, {\em On the recursive structure of Branson's $Q$-curvatures}, Math.
Res. Lett. {\bf 21}, 3 (2014), 495--507. \url{arXiv:1004.1784}

\bibitem [J4]{J2}
A.~Juhl, {\em Explicit formulas for GJMS-operators and $Q$-curvatures}, Geom.
Funct. Anal. {\bf 23}, 4, (2013), 1278--1370. \url{arXiv:1108.0273}

\bibitem [J5]{J3}
A.~Juhl, {\em Heat kernels, ambient metrics and conformal invariants}, Advances
in Math. {\bf 286} (2016), 545--682. \url{arXiv:1411.7851}

\bibitem [J6]{J-bb}
A.~Juhl, {\em On the building block operators of the GJMS operators}.
Unpublished Notes.

\bibitem [KS1]{KS1}
T.~Kobayashi and B.~Speh, {\em Symmetry breaking for representations of rank
one orthogonal groups}, Memoirs of AMS \textbf{238} (2015).
\url{arXiv:1310.3213}

\bibitem [KS2]{KS2}
T.~Kobayashi and B.~Speh, {\em Symmetry breaking for representations of rank one
orthogonal groups II}. Lecture Notes in Mathematics {\bf 2234} (2018), {XV+344}
\url{arXiv:1801.00158}

\bibitem [KKP]{KKP}
T.~Kobayashi, T.~Kubo and M.~Pevzner, {\em Conformal symmetry breaking operators for
differential forms on spheres}, {\em Lecture Notes in Mathematics} {\bf 2170}
(2016), {IX+192}. \url{arXiv:1605.09272}

\bibitem [KOSS]{koss}
T.~Kobayashi, B.~{\O}rsted, P.~Somberg and V.~Sou\v{c}ek, {\em Branching laws for
Verma modules and applications in parabolic geometry. I} Advances in Math. {\bf
285} (2015), 1796--1852. \url{arXiv:1305.6040v1}

\bibitem [KP]{Kobayashi-Pevzner}
T.~Kobayashi and M.~Pevzner, {\em Differential symmetry breaking operators I.
General theory and $F$-method}, Selecta Math. {\bf 12}, 2, (2015), 801--845.
\url{arXiv:1301.2111v4}

\bibitem [M{\O}]{MO}
J.~M\"{o}llers and B.~{\O}rsted, {\em Knapp-Stein type intertwining operators for
symmetric pairs {II}. - The translation principle and intertwining operators for
spinors}. SIGMA {\bf 15} (2019). \url{arxiv:1702.02326}

\bibitem [PP]{PP}
P.~Perry and S.~Patterson, {\em The divisor of Selberg's zeta function for
Kleinian groups}. Duke Math. J. {\bf 106}, 2, (2001), 321--390.
\end{thebibliography}
\end{document}